\documentclass[a4,11pt]{article}
\usepackage[subrefformat=parens]{subcaption}
\usepackage{amssymb}
\usepackage{amstext}
\usepackage{amsmath}
\usepackage{amscd}
\usepackage{amsthm}
\usepackage{amsfonts}
\usepackage{enumerate}
\usepackage{graphicx}
\usepackage{latexsym}
\usepackage{mathrsfs}
\usepackage{mathtools}
\usepackage{cases}
\usepackage{tikz}
\usepackage{caption}
\usepackage[top=3.5cm, bottom=3.5cm, left=3.5cm, right=3.5cm]{geometry}
\usepackage{authblk}
\usepackage{comment}
\usepackage{hyperref}
\newtheorem{theorem}{Theorem}[section]
\newtheorem{corollary}[theorem]{Corollary}
\newtheorem{lemma}[theorem]{Lemma}
\newtheorem{proposition}[theorem]{Proposition}

\newtheorem{remark}[theorem]{Remark}

\newtheorem*{claim*}{Claim}
\theoremstyle{remark}

\numberwithin{equation}{section}

\title{\vspace{-2cm}\textbf{ Local recovery of a piecewise constant anisotropic conductivity in EIT on domains \\ with exposed corners}}

\author[1,a]{\rm Maarten V. de Hoop}
\author[2,b]{\rm Takashi Furuya}
\author[3,c]{\rm Ching-Lung Lin }
\author[4, 5, d]{\rm \\ Gen Nakamura}
\author[6, e]{\rm Manmohan Vashisth}

\affil[1]{{\small Simons Chair in Computational and Applied Mathematics and Earth Science, Rice University, Houston, TX 77005-1892, USA}}
\affil[a]{{\small Email: mdehoop@rice.edu}\vspace{3mm}}

\affil[2]{{\small Education and Research Center for Mathematical and Data Science, Shimane University, Japan}}
\affil[b]{{\small Email: takashi.furuya0101@gmail.com}\vspace{3mm}}

\affil[3]{{\small Department of Mathematics, National Cheng-Kung University, Tainan 701, Taiwan}}
\affil[c]{{\small Email: cllin2@mail.ncku.edu.tw}\vspace{3mm}}

\affil[4]{{\small Department of Mathematics, Hokkaido University, Japan}}

\affil[5]{{\small Research Center of Mathematics for Social Creativity, Research Institute for Electronic Science, Hokkaido University, Japan}}
\affil[d]{{\small Email: nakamuragenn@gmail.com}\vspace{3mm}}

\affil[6]{{\small Department of Mathematics, Indian Institute of Technology, Jammu 181221, India}}
\affil[e]{{\small Email: manmohan.vashisth@iitjammu.ac.in}\vspace{3mm}}

\date{}
\begin{document}

\maketitle

\begin{abstract}
We study the local recovery of an unknown piecewise constant anisotropic conductivity in EIT (electric impedance tomography) on certain bounded Lipschitz domains $\Omega$ in $\mathbb{R}^2$ with corners. The measurement is conducted on a connected open subset of the boundary $\partial\Omega$ of $\Omega$ containing corners and is given as a localized Neumann-to-Dirichlet map. The above unknown conductivity is defined via a decomposition of $\Omega$ into polygonal cells. Specifically, we consider a parallelogram-based decomposition and a trapezoid-based decomposition. We assume that the decomposition is known, but the conductivity on each cell is unknown. We prove that the local recovery is almost surely true near a known piecewise constant anisotropic conductivity $\gamma_0$. We do so by proving that the injectivity of the Fr\'echet derivative $F'(\gamma_0)$ of the forward map $F$, say, at $\gamma_0$ is almost surely true. The proof presented, here, involves defining different classes of decompositions for $\gamma_0$ and a perturbation or contrast $H$ in a proper way so that we can find in the interior of a cell for $\gamma_0$ exposed single or double corners of a cell of $\mbox{\rm supp}H$ for the former decomposition and latter decomposition, respectively. Then, by adapting the usual proof near such corners, we establish the aforementioned injectivity.

\end{abstract}

\date{{\bf Key words}: EIT; domain decomposition; corners; anisotropy; local recovery; finite measurements}

\date{{\bf 2010 Mathematics Subject Classification}: 35R30, 35J25}

\section{Introduction}

Let $\Omega \Subset \mathbb{R}^{2}$ be a simply connected Lipschitz domain with corners occupied by an electric conductive medium. Also, let $\emptyset \neq \Sigma \subset \partial \Omega$ be an open connected set which includes some corners of $\partial\Omega$ and on which we have available boundary measurements. We assume that the electric conductivity, $\gamma$, satisfies $\gamma\in L_+^\infty(\Omega)$, with 
\[
L^{\infty}_{+}(\Omega):=\{ \gamma \in L^{\infty}(\Omega; \mathbb{R}^{2 \times 2})\,: \gamma \geq \delta_{0}I \ \ \mathrm{a.e.} \ \mathrm{in} \ \Omega\,\,\mbox{\rm for a fixed constant $\delta_0>0$}\},
\]
where $I$ is the identity matrix.
Electrical impedance tomography (EIT) is an inverse problem which is stated as follows. Let $\Lambda_\gamma: \dot{H}^{-1/2}_{\diamond}(\overline\Sigma)\rightarrow \overline{H}^{1/2}_{\diamond}(\Sigma)$ be the localized Neumann-Dirichlet map (loc$\_$ND-map) defined as
\[
\Lambda_{\gamma}:\dot{H}^{-1/2}_{\diamond}(\overline\Sigma)\ni f \mapsto u\bigl|_{\Sigma}\in\overline{H}^{1/2}_{\diamond}(\Sigma),
\]
where $u \in H^{1}_{\diamond}(\Omega)$ is the unique solution to the boundary value problem given as
\begin{equation}\label{EIT BVP} 
\nabla \cdot \gamma \nabla u = 0\,\,\text{in}\,\,\Omega, \ \ \partial_\gamma u:=\nu\cdot(\gamma\nabla u) =f\,\,\text{on}\,\,\partial\Omega,
\end{equation}
in which $\nu$ is the unit normal vector of $\partial\Omega$ directed outward.
The function spaces for $\Lambda_\gamma$ are defined by
\[
\dot{H}^{-1/2}_{\diamond}(\overline{\Sigma}):=\left\{ f \in H^{-1/2}_{\diamond}(\partial \Omega)\,: \mathrm{supp}f \subset \overline{\Sigma}\right\},\,\,
\overline{H}^{1/2}_{\diamond}(\Sigma):=\left\{ f \bigl|_{\Sigma}\,: f \in H^{1/2}_{\diamond}(\partial \Omega) \right\}.
\]
where
\[
H^{-1/2}_{\diamond}(\partial \Omega):=\left\{ f \in H^{-1/2}(\partial \Omega)\,: \langle f,1\rangle=0\right\},
\]
\[
H^{1/2}_{\diamond}(\partial \Omega):=\left\{ f \in H^{-1/2}(\partial \Omega)\,: \int_{\partial \Omega}f ds=0 \right\},
\]
\[
H^{1}_{\diamond}(\Omega):=\left\{ u \in H^{1}(\Omega)\,: \int_{\partial \Omega} u ds=0 \right\}.
\]
Here, $\langle \cdot, \cdot \rangle$ denotes the dual bracket in the dual system $\langle H^{-1/2}(\partial \Omega), H^{1/2}(\partial \Omega) \rangle$.
We consider $\Lambda_\gamma$ as boundary measurements. EIT or, in geophysics, the Direct Current (DC) method which concerns determining the conductivity $\gamma$ from $\Lambda_\gamma$, is also known as the Calder\'on problem \cite{calderon2006inverse}.

If the conductivity, $\gamma$, is isotropic, that is, $\gamma = \sigma I$ with a scalar function $\sigma$ on $\Omega$, a vast literature is available and the theory of EIT has achieved a substantial level of completeness, see Uhlmann \cite{Uhlmann2009ElectricalIT}. However, with the counterexample by Tartar \cite{7c36e6a0565e4336a9b6b5bab0f9939b}, the general anisotropic problem still poses several unresolved issues. A principal line of investigation concerning anisotropy in EIT has been of proving uniqueness modulo a change of variables which fixes the boundary \cite{doi:10.1081/PDE-200044485, belishev2003calderon, lassas2003dirichlet, lassas2001determining, lee1989determining, nachman1996global, sylvester1990anisotropic}. In most applications, however, knowledge of position and, hence, coordinates (variables) are important. In this direction, certain, diverse results are available \cite{alessandrini1990singular, Alessandrini2001DeterminingCW,alessandrini2009local,gaburro2009recovering,gaburro2015lipschitz, ikehata2000identification, 7c36e6a0565e4336a9b6b5bab0f9939b, lionheart1997conformal}. In \cite{alessandrini2017uniqueness} a uniqueness result was obtained when the unknown anisotropic conductivity is assumed to be piecewise constant on a given domain decomposition with non-flat separating interfaces. One of the authors of this paper with collaborators also specialized Tartar's counterexample to the case of a half space and constant conductivity thus demonstrating that the non-flatness condition on boundary and interfaces is necessary.

Here, we prove that {\sl{local}} uniqueness and Lipschitz stability are almost surely true for anisotropic piecewise constant conductivities with {\sl{flat}} interfaces defining a domain decomposition can nonetheless be obtained. As a byproduct we can have the probabilistic convergence of the Levenberg-Marquardt iteration scheme and the Landweber iteration scheme for locally recovering the unknown conductivity from the loc$\_$ND-map  (see \cite{stagg, cell1}, and the references therein). For proving them, we introduce a (known) background and an unknown contrast with distinct, particular decompositions of $\Omega$. We refer to subdomains as {\sl cells}. The key contribution of this paper is the exploitation of exposed corners of the support of the unknown anisotropic contrast. Such exposure limits the choice of cells in the decomposition. For example, with triangular decompositions there is not necessarily a properly exposed corner point. In fact, a restriction of decompositions into quadrilateral cells seems essential. 

EIT with conductivities associated with a domain decomposition subjected to appropriate conditions, have been analyzed before in the isotropic case. In the piecewise analytic case, in dimension two, the unique determination was proven in \cite{kohn1985determining}. A conditional Lipschitz stability estimate for identifying an isotropic piecewise constant conductivity with known interfaces from the localized DN-map was given in \cite{alessandrini2005lipschitz}. Here, the interfaces are the inner boundaries following the domain decomposition. Further results were obtained in \cite{alessandrini2017lipschitz, beretta2011lipschitz}, extending the representation to piecewise linear conductivities. The local recovery, in dimension two, of an isotropic piecewise polynomial conductivity on a triangulated domain was analyzed in \cite{lechleiter2008newton}.
A geometric inverse problem that has been studied for piecewise constant perturbations concerns the geodesy ray transform on a two-dimensional compact non-trapping manifold \cite{ilmavirta_lehtonen_salo_2020}.
The authors consider simplices for their decomposition and, thus, manifolds with corners. They exploit the corners in their proof as we do.

EIT is referred to as electric resistivity tomography (ERT) in geophysics. As early as 1920, Conrad Schlumberger \cite{Schlumberger} recognized that anisotropy may affect geological formations' DC electrical properties. Anisotropic effects when measuring electromagnetic fields in geophysical applications have been studied ever since.
From an inverse problems perspective, it is interesting that Maillet and Doll \cite{M-Do} already identified obstructions to recovering an anisotropic resistivity from (boundary) data. Many of the studies of anisotropy in as much as the solutions of the boundary value problem and their probing capabilities are concerned, have been restricted to electrical conductivities (or resistivities) that are piecewise constant while plane layers form the subdomains in a domain decomposition of a half space. That is, flat interfaces separate the subdomains. Yin and Weidelt \cite{yin1999geoelectrical} considered arbitrary anisotropy for the DC-resistivity method in planarly layered media in a geophysics context.

Tiling a domain has a natural link to the domain decomposition of the finite element method. In this paper, we consider two types of quadrilateral domain decompositions forming tilings. One is the parallelogram-based decomposition and the other is a trapezoid-based decomposition. The former one has one exposed corner and the latter one has two exposed corners
in some arrangement.
{\it Having $N$ exposed corner} means that any connected union of quadrilaterals has $N$ corner points which are vertexes of single quadrilateral.
In Figure~\ref{Tiling_quadrilaterals} we show examples of tiling  general domains using both decompositions.
The trapezoidal decomposition fits naturally strata in structural geology in the presence of faults (see \cite{Twiss}). Here, we shed light on in as far as anisotropic conductivities can be almost surely recovered for a domain with aforementioned decompositions.

\begin{figure}[h]
\hspace{0.0cm}
\center
\includegraphics[keepaspectratio, scale=0.4]{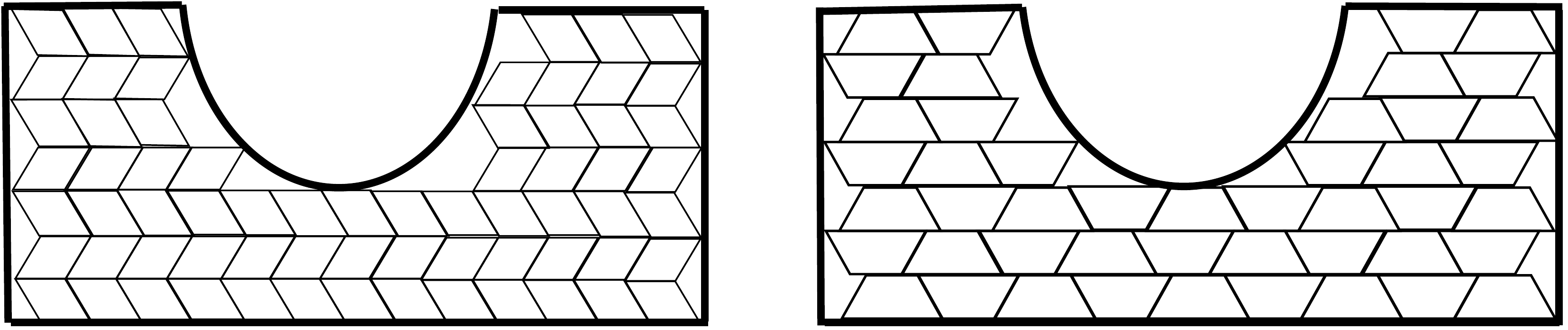}
\caption{
Tiling a general domain using parallelogram (left) and trapezoid (right) based decomposition. 
The black thick curve is the boundary of the domain which is filled by quadrilaterals as much as possible.
In the trapezoid based decomposition, the horizontal array is stacked by shifting in order to have two exposed corners.}
\label{Tiling_quadrilaterals}
\end{figure}

\medskip

\noindent
To state our main results, we define the {\sl forward operator}, $F$, as
\[
F : L^{\infty}_{+}(\Omega) \to \mathcal{L}(\dot{H}^{- 1/2}_{\diamond}(\overline{\Sigma}),\,\,\bar{H}^{1/2}_{\diamond}(\Sigma) )\,\,\mbox{\rm via}\,\,F(\gamma):=\Lambda_{\gamma},
\]
which maps an unknown conductivity $\gamma$ to the measured data $\Lambda_\gamma$. 

\subsection{Parallelogram-based decomposed domain}

First, we introduce the notion of parallelogram-based or skewed-grid-based domain decomposition. For simplicity of describing this, we let
\[
\Omega=[0,1]\times[0,1], \ \ \ \Sigma = \left([0,1]\times\{0\} \right) \cup \left( \{0\}\times[0,1] \right),
\]
with respect to an oblique coordinate with an angle $\theta\in (0,\pi)$.
A parallelogram-based domain, $\Omega$, has the form
\[
\Omega = \bigcup_{C \in \mathcal{P}_{r,\theta}} C,
\]
where $\mathcal{P}_{r,\theta}$ is a set of parallelograms, $C$, given as follows. Let $r \in (0,1]$. Starting from the rhombus $[0,r]\times[0,r]$ , we fill $\Omega$ horizontally by a horizontal array of such rhombuses with side length $r$ as much as possible, and denote the family of these rhombuses by $\mathcal{S}_1$ and their union by $S_1$.
If $T_1:=\left(\Omega\cap([0,1]\times[0,r])\right)\setminus S_1\neq\emptyset$, we add this parallelogram to $\mathcal{S}_1$. We refer to the result as the first stair of an array of parallelograms. Here and hereafter, we abuse the terminologies horizontal and vertical for directions parallel to the axis of the first coordinate and to the axis of the second coordinate of the oblique coordinates, respectively.
Several stairs of such horizontal arrays $\mathcal{S}_j,\,j=1,\cdots J$ including possible $T_j,\,j=1,\dots, J$ can be stacked to fill $\Omega$ as much as possible.

If the parallelograms of $\cup_{j=1}^J\mathcal{S}_j$ cannot fill $\Omega$ entirely, then we divide the remainder at the top vertically by parallelograms with horizontal side length $r$ and vertical one $\tilde{r}$, where $\tilde r$ is the horizontal side length of the parallelogram $T_1$, and define $U:=(1-\tilde r,1)\times(1-\tilde r,1)$ as the rhombus with side length $\tilde r$ in the top right corner of $\Omega$. 
Thus $\mathcal{P}_{r,\theta}$ becomes the set of all of these parallelograms including $\overline U$, see Figure~\ref{parallelogram_division}.

\begin{figure}[h]
\hspace{0.0cm}
\center
\includegraphics[keepaspectratio, scale=0.7]{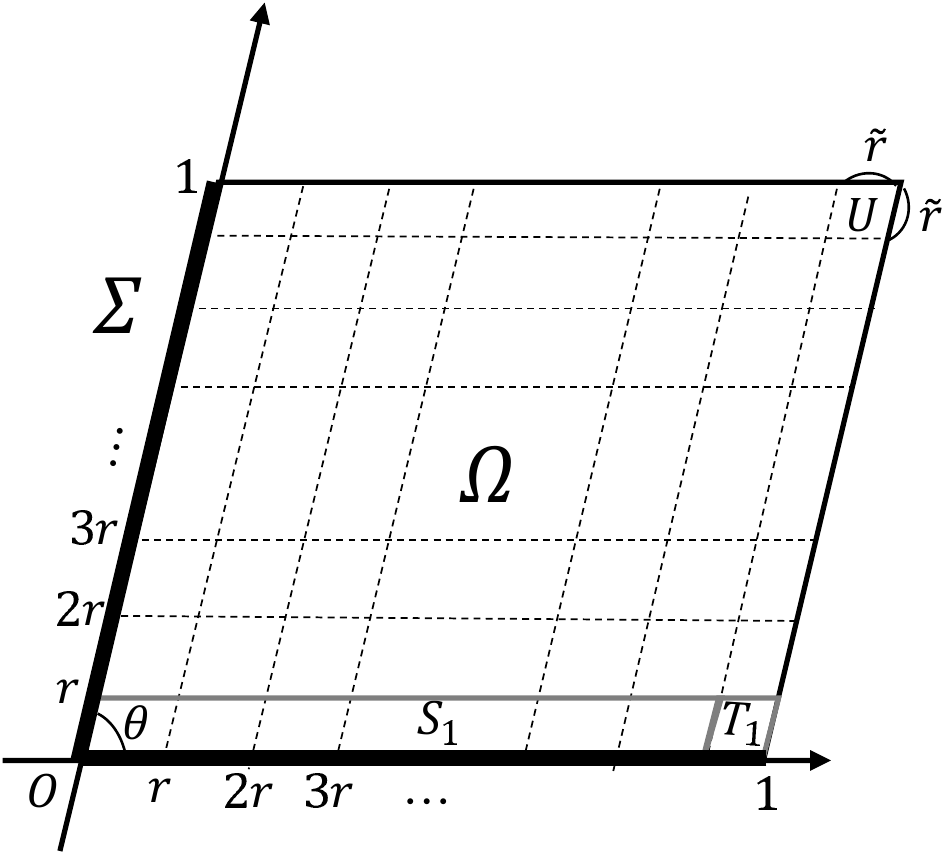}
\caption{Parallelogram decomposition.}
\label{parallelogram_division}
\end{figure}
For $r_{0} \in (0,1]$ and $\theta \in (0, \pi)$, we denote 
\begin{equation}\label{function space B}
V_{\mathcal{P}_{r_0,\theta}}^{c}:=\left\{ \gamma : \Omega \to \mathbb{R}^{2 \times 2} \,:
\mbox{$\gamma$ is $\mathrm{constant}$ for each $C\in \mathcal{P}_{r_0,\theta}$}\right\}.
\end{equation}
Also, for $r \in (0,1]$, $\theta \in (0, \pi)$, and fixing $\{\phi_{c}\}_{C \in \mathcal{P}_{r,\theta}}$ with each $\phi_C\in(0,2\pi]$, we denote 
\begin{equation}\label{function space H}
V_{\mathcal{P}_{r,\theta}, \{\phi_{c}\}}^{c}:=\left\{ H : \Omega \to \mathbb{R}^{2 \times 2} \,:
\begin{array}{cc}
 {}^\forall C\in \mathcal{P}_{r,\theta}, {}^\exists h_{1,c},h_{2,c} \in \mathbb{R}: \\
H\Big|_{C}=R_{\phi_c}^{T} \left(
\begin{array}{cc}
h_{1,c}& 0 \\
0 & h_{2,c} \\
\end{array}
\right) R_{\phi_c}
\end{array}
\right\},
\end{equation}
where $R_{\phi}$ is the rotation matrix with angle $\phi \in (0, 2 \pi]$. 
Further, we denote 
\[
V_{\mathcal{P}_{r_0,\theta}}^{c, +}:= L^{\infty}_{+}(\Omega) \cap V_{\mathcal{P}_{r_0,\theta}}^{c}\,\,\mbox{and}\,\, V_{\mathcal{P}_{r,\theta}, \{\phi_{c}\}}^{c, +}:= L^{\infty}_{+}(\Omega) \cap V_{\mathcal{P}_{r,\theta}, \{\phi_{c}\}}^{c},
\]
from which we take a background conductivity and have a perturbative conductivity, respectively. Here and hereafter, we use the terms perturbation and perturbative for elements of $V_{\mathcal{P}_{r,\theta}, \{\phi_{c}\}}^{c, +}$ and even the difference of their elements.

In addition to these, we define the probability measure $Pr$ on the infinite product space $S:=\prod_{j \in \mathbb{N}}S_{j}$ with $S_1:=(0,1],\,S_{j+1}:=(0,2\pi],\,j\in{\mathbb N}$ as follows. We first let $\{(S_j, \mathcal{F}_{j}, P_j)\}_{j \in \mathbb{N}}$ be the collection of probability spaces defined by
$$
\left\{
\begin{array}{ll}
\mathcal{F}_1:=\mathcal{B}((0,1]),\,\mathcal{F}_{j+1}:=\mathcal{B}((0,2\pi])
\,\,\,&\mbox{for $j\in{\mathbb N}$},\\
P_1:=m,\,\,\,P_{j+1}:=\frac{1}{2\pi}m\,\,\,&\mbox{for $j\in{\mathbb N}$}
\end{array}
\right.
$$
with the Lebesgue measure $m$ on ${\mathbb R}$ and the Borel classes $\mathcal{B}((0,1])$,  $\mathcal{B}((0,2\pi])$ for $(0,1]$, $(0,2\pi]$.
If any subset $A$ of $S$ has the form $A=\prod_{j \in \mathbb{N}}A_{j}\in S$ with $A_j \in \mathcal{F}_{j}$ such that $A_j=S_j$ except for finitely many $j$, we call it a measurable  cylinder. Let $\mathcal{M}$ and $P$ be the collection of all measurable cylinders and the finitely additive set function $P$ on the additive class $\mathcal{M}$ defined by $P(A):=\prod_{j \in \mathbb{N}}P_{j}(A_j)$ for $A \in \mathcal{M}$, respectively.
Then, $P$ admits a unique extension to a probability measure $Pr$ on the $\sigma$-algebra $\sigma(\mathcal{M})$ of $\mathcal{M}$ (see \cite{doi:10.1080/00029890.1996.12004804}).

\medskip

We now state our first main result, which claims probalistic local uniqueness and Lipschitz stability of the inverse problem on parallelogram-based decomposed domains.
\begin{theorem}\label{Probabilistic LLS}
Fix $\theta\in(0,\pi)$, $r_{0} \in (0,1]$, and $\gamma_0 \in \mbox{\rm int}(V_{\mathcal{P}_{r_0,\theta}}^{c,+})$.
For $r \in (0,1]$ and $\vec{\phi}=(\phi_1, \phi_2, ... ) \in \prod_{i \in \mathbb{N}} (0, 2 \pi]$, let $Q^{p}_{r,\vec{\phi}}$ be a proposition stated as follows: 
\begin{equation}
\left\{
\begin{array}{ll}
{}^\exists\epsilon>0:\\
\quad\,\,
{}^\forall \{\phi_{c}\}_{C \in \mathcal{P}_{r,\theta}} \subset \{\phi_{i} \}_{i \in \mathbb{N}_0}, \,\,
{}^\forall\gamma\in\mbox{\rm int}(V_{\mathcal{P}_{r,\theta}, \{\phi_{c}\}}^{c, +})\,\,\mbox{\rm with}\,\,\Vert\gamma-\gamma_0\Vert_\infty\le\epsilon,\\
\quad\,\,\, {}^\exists\delta=\delta(\{\phi_{c}\}, \gamma)>0\,\,\mbox{\rm with}\,\,B_\delta(\gamma)\subset\mbox{\rm int}(V_{\mathcal{P}_{r,\theta}, \{\phi_{c}\}}^{c, +}):
\\
\qquad\,\,\,\,\mbox{\rm such that}\\
\qquad\qquad\Vert\tau-\sigma\Vert_\infty\le C
\Vert F(\tau)-F(\sigma)\Vert,\,\,\tau,\,\sigma\in B_\delta(\gamma)\\
\qquad\qquad\qquad\qquad\quad\,\,\,\mbox{\rm with a constant}\,\,C=C(\{\phi_{c}\}, \gamma)>0,
\end{array}
\right. \label{proposition_LS}
\end{equation}
where $\mbox{\rm int}(E)$ is the interior of a set $E$, $\Vert\tau-\sigma\Vert_\infty$ is the essential supremum of $\tau-\sigma$ in $\Omega$, $\Vert F(\tau)-F(\sigma)\Vert$ is the operator norm of  $F(\tau)-F(\sigma)$ and $B_\delta(\gamma)$ is an open ball with radius $\delta$ centered at $\gamma$. 

Then, for an event $\mathcal{E}^p$ defined by $$\mathcal{E}^{p}=\{ (r,\vec{\phi}) \in (0,1]\times\prod_{i\in \mathbb{N}} (0,2\pi]\,: Q^{p}_{r,\vec{\phi}} \ \text{\rm is true}\},
$$ 
we have
\begin{equation}
Pr(\mathcal{E}^{p})=1. \label{infinite product}
\end{equation}

\end{theorem}
\noindent
The meanings of the logical notations  ${}^\forall$ and ${}^\exists$ used in the above theorem are ``for any'' and ``exists''.
Then, the above theorem means that by specifying the side length $r$ for the perturbation cell and preparing an infinite dimensional vector $\vec{\phi}$ from which we take the rotation angles $\{\phi_{c}\}_{C \in \mathcal{P}_{r,\theta}}$ of the perturbation, we can have the local Lipschitz stability $\Vert\tau-\sigma\Vert_\infty\le C\Vert F(\tau)-F(\sigma)\Vert,\,\,\tau,\,\sigma\in B_\delta(\gamma)$ for almost all such $(r,\vec{\phi})$ in the set $(0,1]\times \prod_{i \in \mathbb{N}} (0, 2 \pi]$.

\subsection{Trapezoid-based decomposed domain}
In this subsection, we introduce another domain decomposition which we call the trapezoid-based domain decomposition. After that we will give our second main result.
\par
To define the decomposition, we first define $\Omega$ and $\Sigma$ as follows. Prepare an even number of isosceles trapezoids with low angle $\theta \in (0, \pi)$ and side length one. Invert half of these trapezoids. Then attach these trapezoids and inverted trapezoids alternatingly to form a parallelogram. Furthermore, slide the inverted trapezoids by $q \in (0,1)$. Then, we define $\Omega$ as the union of these trapezoids and inverted trapezoids (Figure \ref{Shape_omega_trape})
and $\Sigma\subset\partial\Omega$ as the thick black lines in Figure \ref{Shape_omega_trape}.
\par
We also define the following two sets.
Let $\mathcal{L}_{r_0,\theta}$ be the set of cells with lateral side length $r_0$ obtained by dividing $\Omega$ horizontally from the bottom up. 
If $r_0 \notin \mathbb{Q}$, there will be remainders at the top appear with lateral side lengths different from $r_0$. We also include them in the set $\mathcal{L}_{r_0,\theta}$ (see Figure \ref{trapezoid_case1} left). This
$\mathcal{L}_{r_0,\theta}$ gives the set of the background cells.
Let $\mathcal{T}_{r,\theta}$ be the set of isosceles trapezoids with length $r$ obtained by cutting the isosceles trapezoids and the inverted isosceles trapezoids in $\Omega$ horizontally from the bottom up. Again, if $r\notin \mathbb{Q}$, there will be remainders which are also included in $\mathcal{T}_{r,\theta}$ (see the right in Figure \ref{trapezoid_case1}). 
This $\mathcal{T}_{r,\theta}$ gives the set of the perturbative cells. 

\begin{figure}[h]
\hspace{0.0cm}
\center
\includegraphics[keepaspectratio, scale=0.5]{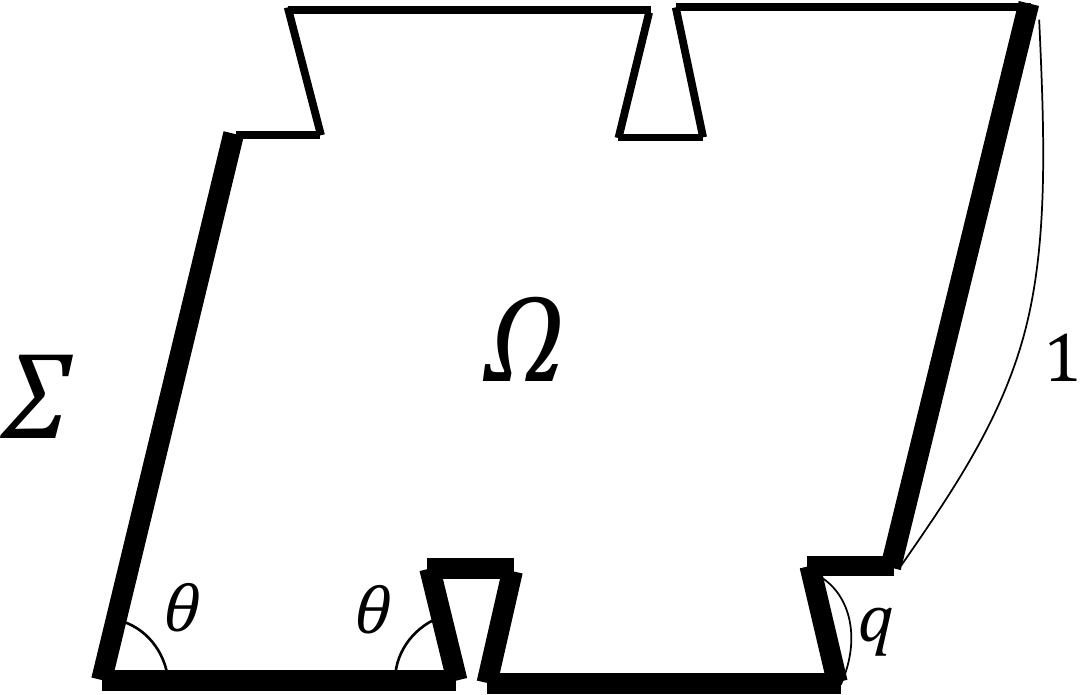}
\caption{Union of trapezoids.} 
\label{Shape_omega_trape}
\end{figure}
\begin{figure}[h]
\hspace{0.0cm}
\center
\includegraphics[keepaspectratio, scale=0.5]{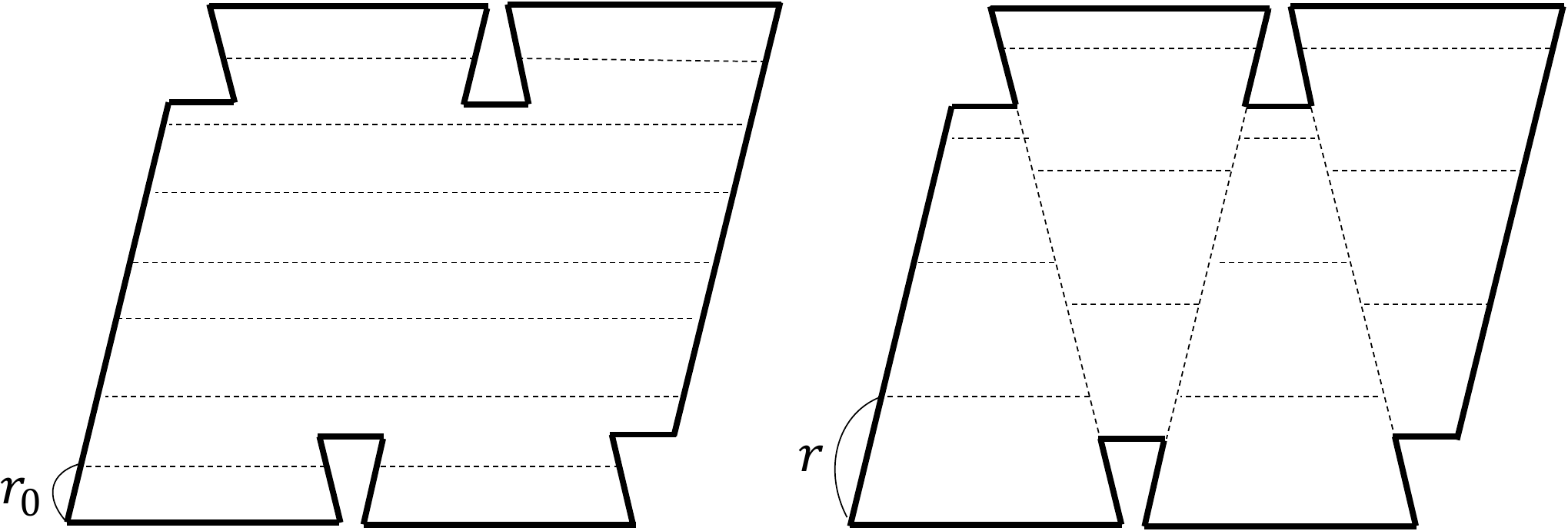}
\caption{Lateral decomposition for background (left), and trapezoid decomposition for perturbation (right).}
\label{trapezoid_case1}
\end{figure}
For $r_{0} \in (0,1]$ and $\theta\in (0,\pi)$, we denote 
\begin{equation}\label{function space B_L}
V_{\mathcal{L}_{r_0,\theta}}^{c}:=\left\{ \gamma : \Omega \to \mathbb{R}^{2 \times 2} \,:
\mbox{$\gamma$ is $\mathrm{constant}$ for each $C \in \mathcal{L}_{r_0,\theta}$}\right\}.
\end{equation}
Also, for $r \in (0,1]$, $0<\theta<\pi$, we denote \begin{equation}\label{function space H_T}
V_{\mathcal{T}_{r,\theta}}^{c}:=\left\{ H : \Omega \to \mathbb{R}^{2 \times 2} \,: \mbox{$H$ is $\mathrm{constant}$ for each $C \in \mathcal{T}_{r,\theta}$}
\right\}. 
\end{equation}
Further, we denote 
\[
V_{\mathcal{L}_{r_0,\theta}}^{c, +}:= L^{\infty}_{+}(\Omega) \cap V_{\mathcal{L}_{r_0,\theta}}^{c}\,\,\mbox{and}\,\, V_{\mathcal{T}_{r,\theta}}^{c, +}:= L^{\infty}_{+}(\Omega) \cap V_{\mathcal{T}_{r,\theta}}^{c}
\]
from which we take a background conductivity and have a perturbative conductivity, respectively. Here and hereafter, we use the terms perturbation and perturbative for elements of
$V_{\mathcal{T}_{r,\theta}}^{c, +}$ and even the difference of their elements. 
We now state our second main result, which claims probalistic local uniqueness and Lipschitz stability of the inverse problem on  trapezoid-based decomposed domains.

\begin{theorem}\label{Probabilistic LLS_trap}
Fix $r_{0}, q \in (0,1]$ satisfying the irrational condition, that is, $\frac{r_0}{q} \not \in{\mathbb Q}$, and $\gamma_0 \in \mbox{\rm int}(V_{\mathcal{L}_{r_0,\theta}}^{c,+})$.
For $r \in (0,1]$ and $\theta \in (0, \pi]$, let $Q^{t}_{r, \theta}$ be a proposition stated as follows: 
\begin{equation}
\left\{
\begin{array}{ll}
{}^\exists\epsilon>0:\\
\quad\,\,
{}^\forall\gamma\in\mbox{\rm int}(V^{c,+}_{\mathcal{T}_{r,\theta}})\,\,\mbox{\rm with}\,\,\Vert\gamma-\gamma_0\Vert_\infty\le\epsilon,\\
\quad\,\,\, {}^\exists\delta=\delta(\gamma)>0\,\,\mbox{\rm with}\,\,B_\delta(\gamma)\subset\mbox{\rm int}(V_{\mathcal{T}_{r,\theta}}^{c, +}):
\\
\qquad\,\,\,\,\mbox{\rm such that}\\
\qquad\qquad\Vert\tau-\sigma\Vert_\infty\le C
\Vert F(\tau)-F(\sigma)\Vert,\,\,\tau,\,\sigma\in B_\delta(\gamma)\\
\qquad\qquad\qquad\qquad\quad\,\,\,\mbox{\rm with a constant}\,\,C=C(\gamma)>0,
\end{array}
\right. \label{proposition_LS_trape}
\end{equation}
Furthermore, let $\mathcal{E}^{t}=\{ (r,\theta) \in (0,1]\times(0,\pi]\,: Q^{t}_{r,\theta} \ \text{\rm is true}\}$ be an event. Then
\[
Pr(\mathcal{E}^{t})=1,
\]
where 
\[
Pr:=\frac{1}{\pi} m,
\]
where $m$ is the two-dimensional Lebesgue measure.
\end{theorem}
\vspace{5pt}
In the parallelogram decomposition, the perturbative function space is defined by specifying the sequence of rotation angles $\{ \phi_c\}$ which correspond to the angles of the symmetry axes of the perturbative conductivity, cell-wise, $H\big|_{C},\,C\in\mathcal{P}_{r,\theta}$ with respect to the axis of the reference coordinates. It is unnecessary to specify these symmetry axes for the trapezoidal decomposition, and, hence, the second main result is stronger than the first one.
This is because we can use two lower corners of each cell of $H$ which are included in the interior of a single cell for the background conductivity, while in the parallelogram division, we can only use a single corner.
Although we gave the decomposition as in Figure \ref{trapezoid_case1}, we can also give the decomposition as in Figure \ref{trapezoid_case2} which allows us to have vertical divisions associated to the conductivities. Indeed, yet other decompositions may be considered, while in this paper we emphasize the principles rather than exhaustively showing the examples of decompositions.

\begin{remark}${}$
By using Theorem 2 in \cite{alberti2020infinitedimensional}, we can show the probabilistic local Lipschitz stability such as Theorems \ref{Probabilistic LLS} and \ref{Probabilistic LLS_trap} with finite measurements, which is the localized Neumann-to-Dirichlet operator projected onto some finite rank space.
\end{remark}

\begin{figure}[h]
\center
\includegraphics[keepaspectratio, scale=0.5]{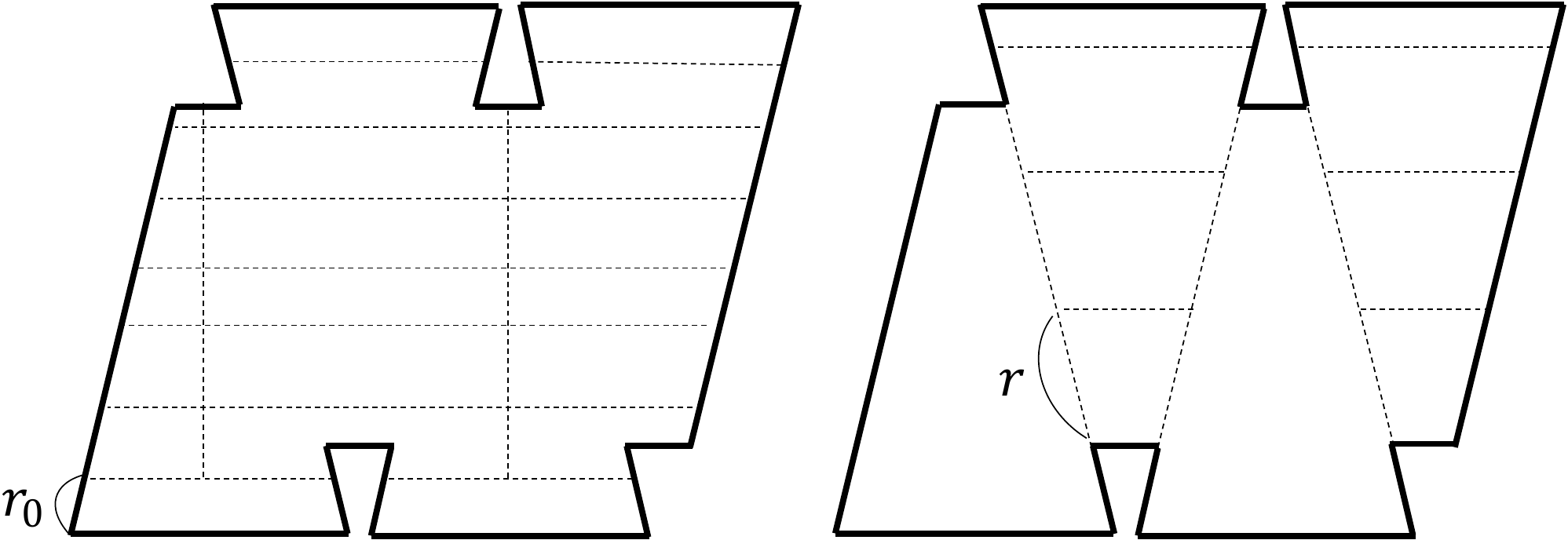}
\caption{Another decomposition for background (left) and  perturbation (right).}
\label{trapezoid_case2}
\end{figure}

\subsection{Local recovery and advantage of using distinct decompositions}

In this subsection, we first point out that we can have a local recovery result for EIT.
More precisely, we have the probabilistic convergence of the Levenberg–Marquardt iteration scheme and Landweber iteration scheme under the settings of the two main results.
This is because Lipschitz stability implies the so-called tangential cone condition, which is known as a sufficient condition for the convergence of these schemes (see e.g., \cite{hanke1997regularizing, kaltenbacher2008iterative}). 

Next, we summarize the key ideas behind the proofs of our main results. We choose $\gamma_0$ and $\gamma$ from different function spaces. More precisely, $\gamma_0\in \mbox{Int}(V_{\mathcal{Q}}^{c,+})$, $\gamma\in \mbox{Int}(V_{\mathcal{R}}^{c,+})$,
where
$$
\left\{
\begin{array}{ll}
\mathcal{Q}=\mathcal{P}_{r_0,\theta},\,\,\mathcal{R}=\mathcal{P}_{r,\theta,\{\phi_c\}}\,\,&\text{for the parallelogram decomposition},\\
\mathcal{Q}=\mathcal{L}_{r_0,\theta},\,\,\mathcal{R}=\mathcal{T}_{r,\theta}
\,\,&\text{for the trapezoidal decomposition}
\end{array}
\right.
$$
with $r,\,r_0$ subject to the {\sl irrationality condition}  $\frac{r}{r_0} \not\in {\mathbb Q}$ and an extra irrationality condition $\frac{q}{r_0}\not \in{\mathbb Q}$ for the trapezoidal decomposition. These irrationality conditions are essential in the following sense. The corner points of the cells in $\mathcal{R}$ are in the interior of a cell in $\mathcal{Q}$ except for corner points on $\partial\Omega$. Then, with the help of the extension argument (see Remark \ref{whole_domain2} given later), the injectivity of the Fr\'echet derivative $F'(\gamma_0)$ of the forward operator $F$ can be proven using the singularity of a fundamental solution for the operator $\nabla\cdot(\gamma_0\nabla\,)$ near the aforementioned corner points. If we don't have such a situation, the structure of the singularity of the fundamental solution becomes exceedingly complicated near those corner points. In view of the irrationality conditions, we state our main results in a  probabilistic framework.

\medskip

The remainder of this paper is organized as follows. In Section~2, we summarize the properties of the forward operator $F$ and prove the probabilistic local Lipschitz stability for the inverse problem assuming the conditional injectivity of Fr\'echet derivative of the forward operator. In Section~3, we prove this injectivity for the parallelogram- and trapezoid-based decompositions as stated in Propositions~\ref{Local_injectivity} and \ref{Local_injectivity_trap}, respectively. Section 4 is devoted to proving two technical lemmas used in the previous section. Then, we end with some concluding remarks and proposing some generalizations of our results.

\section{Local Lipschitz stability}

The purpose of this section is to prove Lemma \ref{Deterministic Lipschitz Stability} given below which states that the conditional injectivity of the Fr\'echet derivative $F^{\prime}$ of forward operator $F$ implies the local Lipschitz stability for the forward operator.
Before the proof, we first review the following statements which can be proved by the same arguments given in sections 2 and 3 of \cite{lechleiter2008newton}.

\begin{lemma}\label{Frechet differentiable}
\begin{description}
\item[(1)] $F$ is  Fr\'echet differentiable at each $\gamma \in \mathrm{int}(L^{\infty}_{+}(\Omega))$ with Fr\'echet derivative $F^{\prime}(\gamma) \in \mathcal{L}( L^{\infty}(\Omega), \mathcal{L}(\dot{H}^{- 1/2}_{\diamond}(\overline{\Sigma}),\bar{H}^{1/2}_{\diamond}(\Sigma) ) ) $ given by
\[
F^{\prime}(\gamma)[H]f:=u^{\prime}\bigl|_{\Sigma},\,\, H\in L^\infty(\Omega)
\]
where $u^{\prime} \in H^{1}_{\diamond}(\Omega)$ is the unique solution for
\begin{equation}
\int_{\Omega} \gamma \nabla u^{\prime} \cdot \nabla \varphi dx = -\int_{\Omega} H \nabla u \cdot \nabla \varphi dx \ \ \mathrm{for} \ \mathrm{all} \ \varphi \in H^{1}_{\diamond}(\Omega), \label{def_of_Fre}
\end{equation}
with the solution $u$ of (\ref{EIT BVP}). 
\item[(2)] There exist some constants $C_{1},...,C_{4} >0$ independent of $\tau, \sigma$ such that 
\begin{equation}
\left\| F(\tau)\right\| \leq C_{1}, \label{Fieq1}
\end{equation}
\begin{equation}
\left\| F(\tau)-F(\sigma) \right\| \leq C_{2} \left\| \tau - \sigma \right\|_{\infty}, \label{Fieq2}
\end{equation}
\begin{equation}
\left\| F(\tau)-F(\sigma) -F^{\prime}(\sigma)[\tau - \sigma] \right\| \leq C_{3} \left\| \tau - \sigma \right\|^{2}_{\infty}, \label{Fieq3}
\end{equation}
\begin{equation}
\left\| F^{\prime}(\tau)-F^{\prime}(\sigma) \right\| \leq C_{4} \left\| \tau - \sigma \right\|_{\infty}\label{Fieq4},
\end{equation}
for $\tau, \sigma \in L^{\infty}_{+}(\Omega)$.
\end{description}
\end{lemma}

\medskip
\vspace{5pt}
We will show the following local Lipschitz stability by assuming the injectivity of the Fr\'{e}chet derivative. 
The idea of the proof is to modify the argument of Theorem 3.4 in \cite{lechleiter2008newton}. 
To proceed further, let $V_{b}$ and $V_{p}$ be subspaces of $L^{\infty}(\Omega; \mathbb{R}^{2 \times 2})$ which correspond to the spaces for the background and perturbation, respectively. We also denote $V_{b}^{+}:=V_{b} \cap L^{\infty}_{+}(\Omega)$ and $V_{p}^{+}:=V_{p} \cap L^{\infty}_{+}(\Omega)$.
\begin{lemma}\label{Deterministic Lipschitz Stability}
Let $\gamma_0 \in V_{b}^{+}$.
Assume that the injectivity of the Fr\'{e}chet derivative holds, that is, the following minimum is positive:
\begin{equation}
\min\{\left\|F^{\prime}(\gamma_{0})[H]\right\| \,: H \in V_{p}, \ \left\| H \right\|_{\infty}=1 \}>0.
\label{inj_Fre_der}
\end{equation}
Then, there exists $\epsilon=\epsilon(\gamma_{0})>0$ such that for any $\gamma \in \mathrm{int}(V^{+}_p)$ with $\left\| \gamma_{0}-\gamma \right\|_{\infty} \leq \epsilon$,
\[
\left\| \tau - \sigma \right\|_{\infty} \leq C \left\|F(\tau) - F(\sigma) \right\|, \ \  \tau, \sigma \in B_{\delta}(\gamma),
\]
holds for some $C=C(\gamma_{0},\gamma)>0$ and $\delta=\delta(\gamma_{0},\gamma)>0$ with $B_{\delta}(\gamma) \subset \mathrm{int}(V_{p}^{+})$.

\end{lemma}
\begin{proof}
Let
\begin{equation}\label{positivity of C}
C_{\gamma_{0}}:=\min\{\left\|F^{\prime}(\gamma_{0})[H]\right\| \,: H \in V_{p}, \ \left\| H \right\|_{\infty}=1 \}>0
\end{equation}
and take $\epsilon, \delta>0$ small enough which will be specified later. Also, let $\tau, \sigma \in B_{\delta}(\gamma)$. Then using \eqref{positivity of C}, we evaluate
\begin{equation}
\begin{split}
& \left\| F^{\prime}(\sigma)[\tau - \sigma]\right\| = \left\| F^{\prime}(\gamma_{0})[\tau - \sigma]+ F^{\prime}(\sigma)[\tau - \sigma]-F^{\prime}(\gamma_{0})[\tau - \sigma]\right\|
\\
& \geq 
\frac{\left\| F^{\prime}(\gamma_{0})[\tau - \sigma] \right\|}{\left\| \tau - \sigma \right\|_{\infty}} \left\| \tau - \sigma \right\|_{\infty} - \left\| F^{\prime}(\sigma)[\tau - \sigma]-F^{\prime}(\gamma_{0})[\tau - \sigma]\right\|
\\
& \geq 
C_{\gamma_{0}}\left\| \tau - \sigma \right\|_{\infty} - C_{4}\underbrace{\left\| \sigma - \gamma_{0} \right\|_{\infty}}_{\leq \left\| \sigma - \gamma \right\|_{\infty}+\left\| \gamma - \gamma_{0} \right\|_{\infty}}\left\| \tau - \sigma \right\|_{\infty}
\\ 
& \geq 
\underbrace{\left( C_{\gamma_{0}} - C_{4}(\delta + \epsilon) \right)}_{=:D=D(\gamma_{0}, \gamma)} \left\| \tau - \sigma \right\|_{\infty}, \label{LLS1}
\end{split}
\end{equation}
where $C_{4}>0$ is the constant given in Lemma \ref{Frechet differentiable} which is independent of $\gamma$ , $\gamma_{0}$, $\tau$ and $\sigma$. Here we choose $\epsilon=\epsilon(\gamma_{0})>0$ and  $\delta=\delta(\gamma, \gamma_{0})>0$ to satisfy $C_{\gamma_{0}} - C_{4}\epsilon > 0$
and $D=D(\gamma_{0}, \gamma)=C_{\gamma_{0}} - C_{4}(\delta+\epsilon)> 0$, respectively.
 
By (\ref{Fieq3}) and (\ref{LLS1}), we have
\begin{equation}
\begin{split}
& \left\| F(\tau)-F(\sigma) -F^{\prime}(\sigma)[\tau - \sigma] \right\| \leq C_{3} \frac{\left\|F^{\prime}(\sigma)[\tau - \sigma] \right\|}{\left\|F^{\prime}(\sigma)[\tau - \sigma] \right\|} \left\| \tau - \sigma \right\|^{2}_{\infty}  
\\
& \leq 
\frac{C_3}{D(\gamma_{0}, \gamma)} \left\|F^{\prime}(\sigma)[\tau - \sigma] \right\|\left\| \tau - \sigma \right\|_{\infty}, \label{LLS2}
\end{split}
\end{equation}
where $C_{3}>0$ is the constant given in Lemma \ref{Frechet differentiable} which is independent of $\gamma$ , $\gamma_{0}$, $\tau$ and $\sigma$.
From (\ref{LLS2}) we obtain
\[
\begin{split}
& \left\| F^{\prime}(\sigma)[\tau - \sigma] \right\| - \left\| F(\tau)-F(\sigma)\right\|
\\
& \leq
\left\| F(\tau)-F(\sigma) -F^{\prime}(\sigma)[\tau - \sigma] \right\| \leq \frac{C_3}{D(\gamma_{0}, \gamma)} \left\|F^{\prime}(\sigma)[\tau - \sigma] \right\| \underbrace{\left\| \tau - \sigma \right\|_{\infty}}_{\leq 2\delta},
\end{split}
\]
which implies that
\begin{equation}
\underbrace{\left(1-\frac{2C_3 \delta}{D(\gamma_{0}, \gamma)} \right)}_{=:E=E(\gamma_{0}, \gamma)} \left\|F^{\prime}(\sigma)[\tau - \sigma] \right\| \leq \left\| F(\tau) - F(\sigma) \right\|. \label{LLS3}
\end{equation}
If necessary, we choose $\delta=\delta(\gamma, \gamma_{0})>0$ smaller such that $E(\gamma_{0}, \gamma) > 0$. 
Therefore by (\ref{LLS1}) and (\ref{LLS3}), we have  
\[
\left\| \tau - \sigma \right\|_{\infty} \leq \frac{1}{D(\gamma_{0}, \gamma)E(\gamma_{0}, \gamma)} \left\|F(\tau) - F(\sigma) \right\|.
\]
\end{proof}

\section{Proof of conditional injectivity of the Fr\'{e}chet derivative} 
\label{Core lemma}

The task in this section is to analyze when we can have the conditional injectivity of the Fr\'echet derivative $F^{\prime}$ of the forward operator $F$ for the parallelogram and trapezoidal decompositions. By applying the two core Lemma \ref{Blow up of the integral} and Lemma \ref{Blow up of the integral_trap} given later in Section 4, we can achieve this task. As a result, it can be seen that our first two main results hold.

\subsection{Parallelogram decomposition}\label{Parallelogram decomposition}
\begin{proposition} \label{Local_injectivity}
Let $\theta\in(0,\pi)$, $r_{0} \in (0,1]$, and $\gamma_0 \in \mbox{\rm int}(V_{\mathcal{P}_{r_0,\theta}}^{c,+})$.
Assume that $r,r_{0} \in (0,1]$ satisfy the irrational condition, that is, $\frac{r}{r_0}\not\in{\mathbb Q}$.
Also, assume that $\vec{\phi}=(\phi_1, \phi_2, ... ) \in \prod_{i \in \mathbb{N}} (0, 2 \pi]$ satisfies
\begin{equation}
2 \phi_{i} + \alpha_{c} \neq 0, \pi, 2\pi, 3\pi, 4\pi, 5\pi, 
\label{assumption_for_phi}
\end{equation}
for all $i \in \mathbb{N}$ and $C \in \mathcal{P}_{r_0,\theta}$, where some $\alpha_{c}$ depends on $\gamma_{0}\bigl|_{C}$ (for the definition of $\alpha_c$, see (\ref{a_b}), (\ref{determinant}), (\ref{p_q}) in the proof).
Then, we have
\begin{equation}\label{injectivity}
\min\{\left\|F^{\prime}(\gamma_{0})[H]\right\| \,: H \in V_{\mathcal{P}_{r,\theta}, \{\phi_{c}\}}^{c}, \ \left\| H \right\|_{\infty}=1 \}>0.
\end{equation}
\end{proposition}
\begin{proof}
Suppose \eqref{injectivity} does not hold. 
Since $V_{\mathcal{P}_{r,\theta}, \{\phi_{c}\}}^{c}$ is a finite dimensional linear space, there exists $H \in V_{\mathcal{P}_{r,\theta}, \{\phi_{c}\}}^{c}$ with $\left\| H \right\|_{\infty}=1$ such that  
\[
F^{\prime}(\gamma_{0})[H]f=u^{\prime}\bigl|_{\Sigma}= 0,\,\,f \in \dot{H}^{- 1/2}_{\diamond}(\overline{\Sigma})
\]
where $u^{\prime}$ is solution of (\ref{def_of_Fre}) with $\gamma=\gamma_{0}$. This implies that
\begin{equation}
\begin{split}
& - \int_{\Omega} H \nabla u \cdot \nabla u dx = \int_{\Omega} \gamma_{0} \nabla u^{\prime} \cdot \nabla u dx =
\int_{\Omega} \gamma_{0} \nabla u  \cdot  \nabla u^{\prime}  dx
\\
& = \int_{\partial \Omega} \nu \cdot (\gamma_{0} \nabla u ) u^{\prime} ds - \int_{\Omega} \nabla \cdot (\gamma_{0} \nabla u ) u^{\prime} dx = 0, \label{integral identity}
\end{split}
\end{equation}
where $u \in H^{1}_{\diamond}(\Omega)$ is the unique solution for (\ref{EIT BVP}) with $\gamma=\gamma_{0}$. 
\par
Since $H \neq 0$ and $H$ is a constant $2 \times 2$ matrix for each parallelogram cell $C \in \mathcal{P}_{r,\theta}$,
$\mathrm{supp}H$ consists of several parallelogram cells in $\Omega$ and its boundary $\partial\, \mathrm{supp}H$ consists of edges and vertices of parallelogram cells.
By extending the domain $\Omega$ (see Remark \ref{whole_domain2} and Figure \ref{extend_domain}), we can assume that 
\begin{equation}
\mathrm{supp}H \cap \Sigma = \emptyset. \label{extension}
\end{equation}
Then, there exists an exposed corner point $x_{0}$ at the lower left of some cell $C_{H} \in \mathcal{P}_{r,\theta}$ (see Figure \ref{exposed_corner}).
\begin{figure}[h]
\hspace{0.0cm}
\center
\includegraphics[keepaspectratio, scale=0.6]{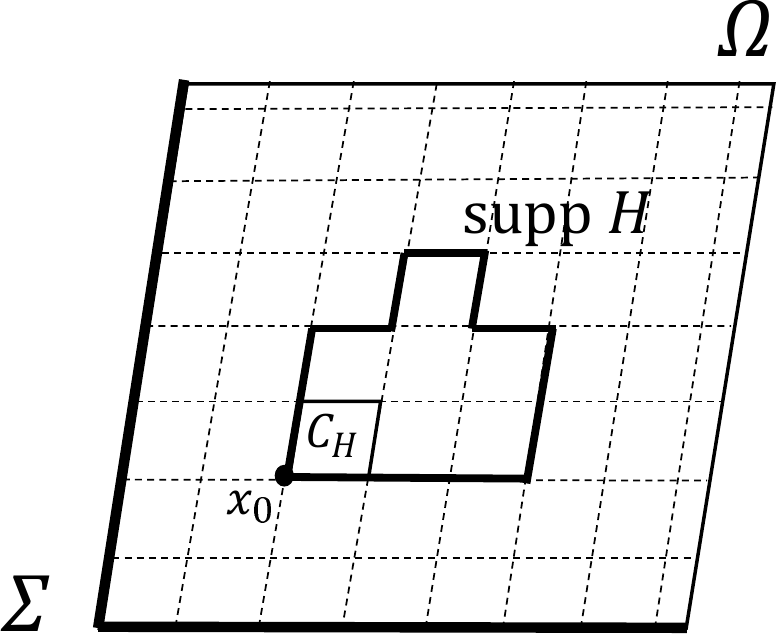}
\caption{Exposed corner.}
\label{exposed_corner}
\end{figure}
Let us write $H$ in the form (see definition of the function space (\ref{function space H}))
\begin{equation}
H = R_{\phi_{C_{H}}}^{T} \left(
\begin{array}{cc}
h_{1}& 0 \\
0 & h_{2} \\
\end{array}
\right)R_{\phi_{C_{H}}}\neq 0 \ \ \mathrm{on} \ \ C_{H}. \label{assumption}
\end{equation}
Then, one of the constants $h_{i}$ is not zero.

Note that $\gamma_0$ is a constant $2 \times 2$ matrix near the corner $x_{0}$ because all corners of the cells in $\mathcal{P}_{r_{0}, \theta}$ except those on $\partial \Omega$ differ from those of cells in $\mathcal{P}_{r, \theta}$. In fact, otherwise, there exist $N,M \in \mathbb{N}$ such that $rN=r_{0}M$ which contradicts to $\frac{r}{r_{0}} \notin \mathbb{Q}$.
We denote by $C_{\gamma}$ such a cell for $\gamma_0$ including $x_0$ in its interior.
\par
Now we consider the transformations given in Figure \ref{transform}.
\begin{figure}[h]
\hspace{0.0cm}
\center
\includegraphics[keepaspectratio, scale=0.6]{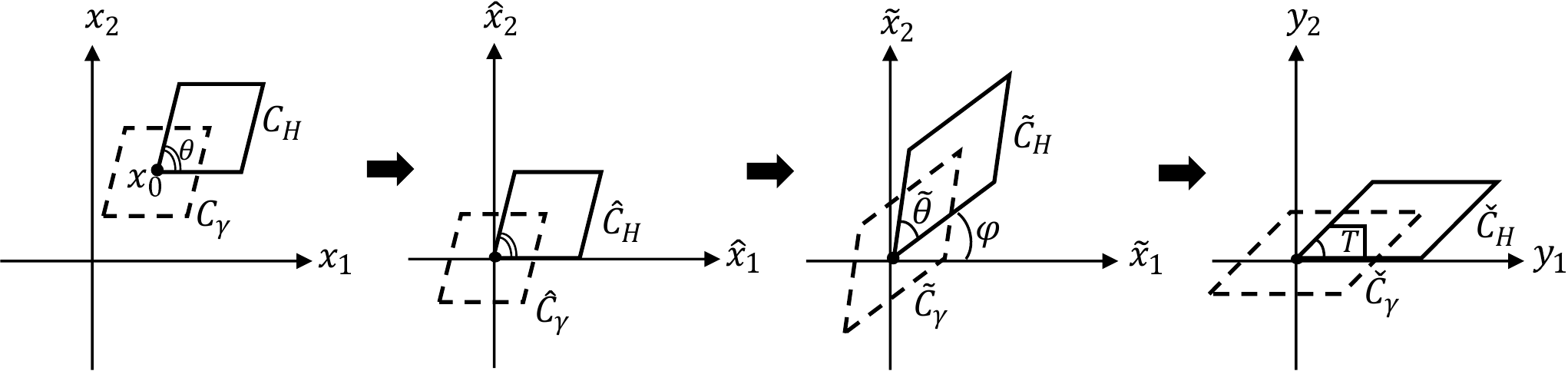}
\caption{Transformations.}
\label{transform}
\end{figure}
By the change of variables $x=\hat{x}+x_0$, we have from (\ref{integral identity})
\begin{equation}
\int_{\hat{\Omega}} \hat{H} \nabla \hat{u} \cdot \nabla \hat{u}\, d\hat{x} = 0, \label{integral identity2}
\end{equation}
where $\hat{u}(\hat x):=u(x-x_0) \in H^{1}_{\diamond}(\hat{\Omega})$ is the unique solution of
\[
\nabla \cdot \hat \gamma_0 \nabla \hat u = 0\,\,\text{in}\,\,\hat \Omega, \ \ \partial_{\hat \gamma_{0}} \hat u=\hat f\,\,\text{on}\,\,\partial \hat\Omega, \,\,\hat{f} \in \dot{H}^{- 1/2}_{\diamond}(\overline{\hat{\Sigma}}),
\]
where we have denoted 
\[
\hat{\Omega}:=\Omega-x_0, \ \  \hat{\Sigma}:=\Sigma-x_0, \ \
\hat{C}_{H}:=C_{H}-x_0, \ \ 
\hat{C}_{\gamma}:=C_{\gamma}-x_0, 
\]
\[
\hat{H}:=H(\cdot + x_0), \ \
\hat{\gamma}_{0}:=\gamma_{0}(\cdot + x_0).
\]
Since $\hat{\gamma}_{0}(0)$ is a positive definite matrix, there exists a rotation matrix $R_{\psi}$ with an angle $\psi \in (0, 2 \pi]$ and a diagonal matrix $D=\mathrm{diag}(d_1, d_2)$ with $d_1, d_2 >0$ such that 
\[
\hat{\gamma}_{0}(0)=R^{T}_{\psi}DR_{\psi}.
\]
Then, by the change of variables given as $\hat{x}=(D^{1/2}R_{\psi})^{T}\tilde{x}$, we have from (\ref{integral identity2})
\begin{equation}
\int_{\tilde{\Omega}} \tilde{H} \nabla \tilde{u} \cdot \nabla \tilde{u} d\tilde{x} = 0, \label{integral identity3}
\end{equation}
where $\tilde{u}(\tilde x):=\hat u(D^{-1/2}R_{\psi}\hat x) \in H^{1}_{\diamond}(\tilde{\Omega})$ is the unique solution for 
\[
\nabla \cdot \tilde \gamma_0 \nabla \tilde u = 0\,\,\text{in}\,\,\tilde \Omega, \ \ \partial_{\tilde \gamma_{0}} \tilde u=\tilde f\,\,\text{on}\,\,\partial \tilde\Omega, \,\,\tilde{f} \in \dot{H}^{- 1/2}_{\diamond}(\overline{\tilde{\Sigma}}),
\]
and we have denoted
\[
\tilde{\Omega}:=D^{-1/2}R_{\psi}(\hat{\Omega}), \ \ 
\tilde{\Sigma}:=D^{-1/2}R_{\psi}(\hat{\Sigma}), \ \
\tilde{C}_{H}:=D^{-1/2}R_{\psi}\hat{C}_{H}, \ \ 
\tilde{C}_{\gamma}:=D^{-1/2}R_{\psi}\hat{C}_{\gamma}, 
\]
\[
\tilde{H}:=D^{-1/2}R_{\psi}\hat{H}(D^{-1/2}R_{\psi})^{T}, \ \ 
\tilde{\gamma}:=D^{-1/2}R_{\psi}\hat{\gamma} (D^{-1/2}R_{\psi})^{T}.
\]
Here we remark that 
\[
\tilde{\gamma}_{0}(0)=D^{-1/2}R_{\psi}R^{T}_{\psi}DR_{\psi}R^{T}_{\psi}D^{-1/2}=I.
\]

To proceed further, let $\varphi \in (0, 2\pi]$ be the angle between the axis $\tilde{x}_1$ and \linebreak $D^{-1/2}R_{\psi}\left(
\begin{array}{c}
1 \\
0
\end{array}
\right)$, and let $\tilde{\theta} \in(0, \pi]$ be the angle between $D^{-1/2}R_{\psi}\left(
\begin{array}{c}
1 \\
0
\end{array}
\right)$ and 
\linebreak
$D^{-1/2}R_{\psi}\left(
\begin{array}{c}
k \\
1
\end{array}
\right)$ with $k:= \frac{1}{\mathrm{tan} \theta}$.
\newline
By denoting $d:=\left\|D^{-1/2}R_{\psi}\left(
\begin{array}{c}
1 \\
0
\end{array}
\right) \right\|$, we have
\begin{equation}
\left(
\begin{array}{c}
\mathrm{cos}\varphi \\
\mathrm{sin}\varphi
\end{array}
\right)
=
d^{-1}
\left(
\begin{array}{c}
d_{1}^{-1/2}\mathrm{cos}\psi \\
d_{2}^{-1/2}\mathrm{sin}\psi
\end{array}
\right). \label{angle_varphi}
\end{equation}
Furthermore, by the change of variables given as $\tilde{x}=R_{\varphi}y$, we have from (\ref{integral identity3})
\begin{equation}
\int_{\check{\Omega}} \check{H} \nabla \check{u} \cdot \nabla \check{u} dy = 0, \label{integral identity4}
\end{equation}
where $\check{u}(\check x):=\tilde u(R_{\varphi}^{T}\tilde x) \in H^{1}_{\diamond}(\check{\Omega})$ is the unique solution for 
\begin{equation}
\nabla \cdot \check \gamma_0 \nabla \check u = 0\,\,\text{in}\,\,\check \Omega, \ \ \partial_{\check \gamma_{0}} \check u=\check f\,\,\text{on}\,\,\partial \check\Omega, \,\,\check{f} \in \dot{H}^{- 1/2}_{\diamond}(\overline{\check{\Sigma}}), \label{reduced equation}
\end{equation}
and we have denoted
\[
\check{\Omega}:=R_{\varphi}^{T}(\tilde{\Omega}), \ \
\check{\Sigma}:=R_{\varphi}^{T}(\tilde{\Sigma}), \ \
\check{C}_{H}:=R_{\varphi}^{T}\tilde{C}_{H}, \ \ 
\check{C}_{\gamma}:=R_{\varphi}^{T}\tilde{C}_{\gamma}, 
\]
\[
\check{H}:=R_{\varphi}^{T}\tilde{H}R_{\varphi}, \ \ 
\check{\gamma}:=R_{\varphi}^{T}\tilde{\gamma} R_{\varphi}.
\]
Here we remark that 
\begin{equation}
\check{\gamma}_{0}(0)=R_{\varphi}^{T}IR_{\varphi}=I \label{freezing} 
\end{equation}
and the angle of the corner at the origin is $\tilde \theta$ (see Figure \ref{transform}).
\par
We apply the Runge approximation theorem given in Lemma 4.1 of \cite{carstea2018uniqueness} by taking $\Omega_1,\,\Gamma_1$, and $\Omega_2$ there as follows. Namely, $\Omega_1=\check{\Omega}$, $\Gamma_1 = \check \Sigma$, and $\Omega_2 \subset \Omega_1$ is chosen such that $\mathrm{supp} \check{H} \subset \Omega_2$, $\eta \notin \Omega_2$, $\partial \Omega_1 \setminus \partial \Omega_2 = \Gamma_1$, $\partial \Omega_1 \cap \partial \Omega_2 = \partial \Omega_1 \setminus \Gamma_1$ (see Figure \ref{Runge_approx}). Here $\eta\sim0$ is the singular point of a fundamental solution $E_{\eta}^{\check{\gamma}_{0}}$ for the partial differential operator $\nabla\cdot(\check{\gamma}_{0}\nabla)$ in $\mathbb{R}^{2}$.
\begin{figure}[h]
\hspace{0.0cm}
\center
\includegraphics[keepaspectratio, scale=0.6]{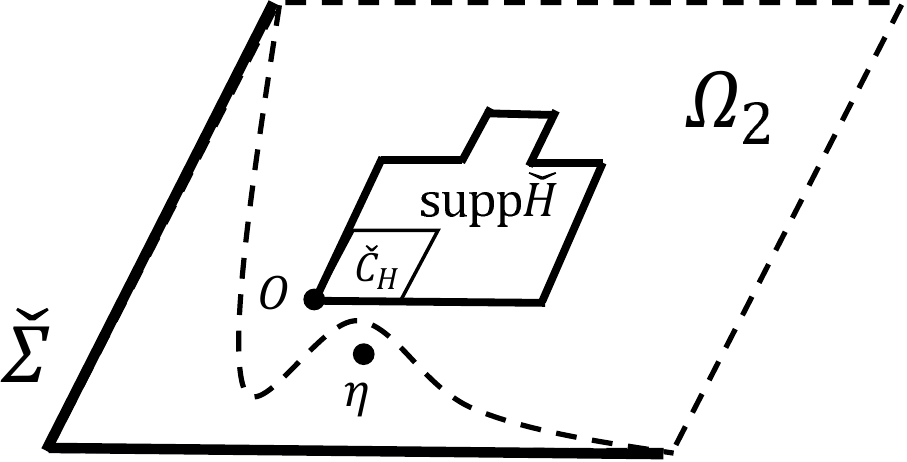}
\caption{Runge approximation.}
\label{Runge_approx}
\end{figure}
Then, there exists $\{\check{f}_{n}\}_{n=1}^{\infty} \subset \dot{H}^{-1/2}_{\diamond}(\overline{\check{\Sigma}})$ such that the sequence of solutions $\check{u}_{n},\,n\in{\mathbb N}$ of \eqref{reduced equation} with $\check f=\check f_n,\,n\in{\mathbb N}$ approximates $E_{\eta}^{\check{\gamma}_{0}}$ as $n \to \infty$.
By letting $n \to \infty$ in (\ref{integral identity4}) with $\check{u}=\check{u}_{n}$, we have 
\[
0 = \int_{\mathrm{supp}\check{H}}\check{H}\nabla E_{\eta}^{\check{\gamma}_{0}} \cdot \nabla E_{\eta}^{{\check\gamma}_{0}} dy.
\]
Note that $\mathrm{supp}\check{H}$ has a small closed trapezoid $T$ with the angle $\tilde \theta$ at the origin such that $T \subset \check{C}_{H} \cap \check{C}_{\gamma}$ (see Figure \ref{transform}, right), and $\check \gamma$ is the identity matrix in the neighborhood $V \Subset \check{C}_{\gamma}$ of $T$ (see (\ref{freezing})).
Then, 
\[
E_{\eta}^{\check{\gamma}_{0}} - E_{\eta}^{0} \in H^{1}(V), \ \ \eta \in V,
\]
where $E_{\eta}^{0}$ is a fundamental solution for the partial differential operator freezing the coefficient at the origin $\nabla\cdot(\check{\gamma}_{0}(0)\nabla)$, which is the fundamental solution for $\Delta$, that is
\[
E_{\eta}^{0}(y) = \frac{1}{2\pi} \log|y-\eta|.
\]
Then, we have
\begin{equation}
\begin{split}
& 0 = 
\int_{T}\check{H}\nabla E_{\eta}^{\check{\gamma}_{0}} \cdot \nabla E_{\eta}^{\check{\gamma}_{0}} dy+\underbrace{\int_{\mathrm{supp}\check{H}\setminus T}\check{H}\nabla E_{\eta}^{\check{\gamma}_{0}} \cdot \nabla E_{\eta}^{\check{\gamma}_{0}} dy}_{=O(1) \ as \ \eta \to 0}
\\
& =\int_{T}\check{H}\nabla (E_{\eta}^{0}+E_{\eta}^{\check{\gamma}_{0}} - E_{\eta}^{0}) \cdot \nabla (E_{\eta}^{0}+E_{\eta}^{\check{\gamma}_{0}} - E_{\eta}^{0}) dy + O(1)
\\
& = \frac{1}{(2 \pi)^2}\underbrace{\int_{T}\check{H}\nabla_{y} \log|y-\eta| \cdot \nabla_{y}\log|y-\eta| dy}_{=:I(\eta)} + O(1). \label{key integral}
\end{split}
\end{equation}
for all $\eta \in V \setminus T$, which implies that
\[
\sup_{\eta \in V \setminus T}|I(\eta)| < \infty.
\]
Using Lemma \ref{Blow up of the integral} given later in Section 4, we have
\begin{equation}
\left \{
\begin{array}{l}
\check{h}_{11}+\check{h}_{22}=0\\
\tilde k\check{h}_{11} + \check{h}_{12}=0,
\end{array}
\right. \label{condition1}
\end{equation}
where 
\[
\check{H}\bigl|_{\check C_H} = \left(
\begin{array}{cc}
\check{h}_{11}& \check{h}_{12} \\
\check{h}_{12}& \check{h}_{22} \\
\end{array}
\right),\,\,\tilde k := \frac{1}{\mathrm{tan} \tilde \theta}.
\]
By the transformation we have introduced so far, the matrix $\check{H}\bigl|_{\check C_H}$ has the form 
\[
\check{H}\bigl|_{\check C_H}=R_{\varphi}^{T}D^{-1/2}R_{\psi}\hat{H}\bigl|_{\hat C_H}(R_{\varphi}^{T}D^{-1/2}R_{\psi})^{T}.
\]
By direct computation and using (\ref{angle_varphi}), we have  
\[
R_{\varphi}^{T}D^{-1/2}R_{\psi} = 
d
\left(
\begin{array}{cc}
1 & a \\
0 & b \\
\end{array}
\right),
\]
where
\begin{equation}
a:=\frac{d_2-d_1}{2d^{2}}\mathrm{sin}2\psi, \ \ \ b:=\frac{d^{-1/2}_{1}d^{-1/2}_{2}}{d^{2}}. \label{a_b}
\end{equation}
Using (\ref{assumption}), we have
\[
\hat{H}\bigl|_{\hat C_H} = R_{\phi_{C_{H}}}^{T} \left(
\begin{array}{cc}
h_{1}& 0 \\
0 & h_{2} \\
\end{array}
\right)R_{\phi_{C_{H}}}.
\]
By direct computation, the condition (\ref{condition1}) for $\check{h}_{ij}$ is translated into the condition for $h_i$ as
\begin{equation}
\left \{
\begin{array}{l}
\left\{\mathrm{cos}^{2}\phi - a \mathrm{sin}2\phi+(a^2+b^2)\mathrm{sin}^{2}\phi \right\}h_{1}
\\ \hspace{2cm}
+\left\{\mathrm{sin}^{2}\phi + a \mathrm{sin}2\phi+(a^2+b^2)\mathrm{cos}^{2}\phi \right\}h_{2}=0, 
\vspace{2mm}
\\
\left\{-\frac{1}{2}\mathrm{sin}2\phi + (a- \tilde kb) \mathrm{sin}^{2}\phi\right\}h_{1}+\left\{\frac{1}{2}\mathrm{sin}2\phi + (a-\tilde kb) \mathrm{cos}^{2}\phi \right\}h_{2}=0.
\end{array}
\right. \label{condition2}
\end{equation}
By direct computation, the determinant of equation (\ref{condition2}) is given by
\begin{equation}
\mathrm{(Determinant)}=p\,\mathrm{sin}2\phi+q\,\mathrm{cos}2\phi=\sqrt{p^2+q^2}\mathrm{sin}(2\phi+\alpha), \label{determinant}
\end{equation}
where
\begin{equation}
p:=-\frac{1}{2}a^2+\frac{1}{2}b^2+\tilde kab + \frac{1}{2}, \ \ \ q:=a-\tilde kb, \label{p_q} 
\end{equation}
and $\alpha \in (-\pi, \pi]$ depending on $\gamma_{0}\bigl|_{C_{\gamma_0}}$ is the angle corresponding to the composition of trigonometric functions.
By assumption (\ref{assumption_for_phi}), (Determinant) is not zero. 
Therefore, we conclude that 
\[
h_1=h_2=0,
\]
which contradicts to (\ref{assumption}).
Hence, Proposition \ref{Local_injectivity} has been proved.
\end{proof}
\begin{remark}
In terms of the original coordinates, the trace zero condition for $\check{H}\bigl|_{\check C_H}$ can be computed as follows:
\[
\begin{split}
& 0 = \mathrm{Tr} [\check{H}\bigl|_{\check C_H}] 
= \mathrm{Tr}[R_{\varphi}^{T}D^{-1/2}R_{\psi}\hat{H}\bigl|_{\hat C_H}(R_{\varphi}^{T}D^{-1/2}R_{\psi})^{T}]
\\
& = \mathrm{Tr}[\hat{H}\bigl|_{\hat C_H}R_{\varphi}^{T}D^{-1}R_{\psi}] 
= \mathrm{Tr}[\hat{H}\bigl|_{\hat C_H} \hat\gamma_{0}(0)^{-1}]
= \mathrm{Tr}[H\bigl|_{C_H} \gamma_{0}(x_0)^{-1}].
\end{split}
\]
We note that this condition can be derived by approaching $\eta$ to no matter which edge
of the cell $\check{C}_H$ excluding $0$ in Figure \ref{Runge_approx}. 
\end{remark}
\begin{remark}(extention argument)\label{whole_domain2}
In the proof, we have assumed (\ref{extension}) which can be justified as follows by just using the following condition:
\[
\int_{\Omega}H \nabla u  \cdot \nabla u dx = 0,
\]
where $H\in L^\infty(\Omega)$ and $u \in H^{1}_{\diamond}(\Omega)$ is the unique solution of
\begin{equation}
\nabla \cdot \gamma_0 \nabla  u = 0\,\,\text{in}\,\,\Omega, \ \ \partial_{ \gamma_{0}}  u= f\,\,\text{on}\,\,\partial \Omega, \,\,{f} \in \dot{H}^{- 1/2}_{\diamond}(\overline{\Sigma}). \label{rem_sol_Omega}
\end{equation}
In fact, let $\Omega_{E} \subset \mathbb{R}^{2}$ be the parallelogram extension of $\Omega$ with $\Omega \Subset \Omega_{E}$ (see Figure \ref{extend_domain}), and let $\gamma_{0,E}$ be the extension of $\gamma_{0}$ to $\Omega_{E}$ with $\gamma_{0,E}|_{\Omega}=\gamma_{0}$, and let $H_E$ be the zero extension of $H$ to $\Omega_{E}$. 
\begin{figure}[h]
\hspace{0.0cm}
\center
\includegraphics[keepaspectratio, scale=0.5]{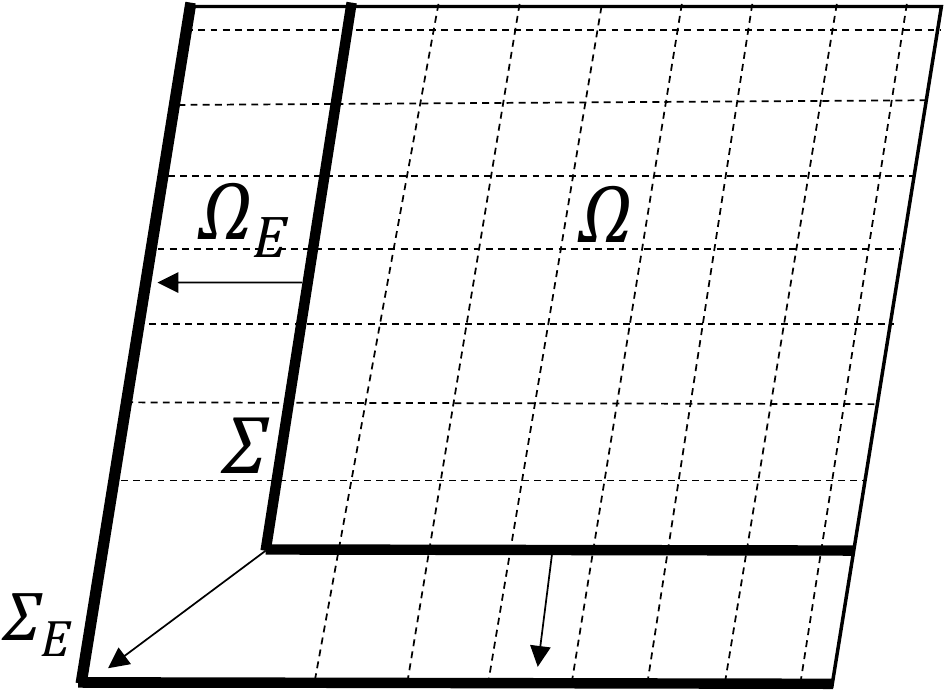}
\caption{Extension of domain.}
\label{extend_domain}
\end{figure}
Also, let $w=w^g \in H^{1}_{\diamond}(\Omega_E)$ be the solution of 
\begin{equation}
\nabla \cdot \gamma_{0,E} \nabla w = 0\,\,\text{in}\,\,\Omega_E, \ \ \partial_{\gamma_{0,E}} w= g\,\,\text{on}\,\,\partial\Omega_E.\label{rem_sol_tilde_Omega}
\end{equation}
for any given $ g \in \dot{H}^{- 1/2}_{\diamond}(\overline{\Sigma_{E}})$.
Then, the restriction $w|_{\Omega}$ on $\Omega$ satisfies (\ref{rem_sol_Omega}) with $f=\partial_{\gamma_{0,E}} w\big|_{\partial\Sigma} \in \dot{H}^{- 1/2}_{\diamond}(\overline{\Sigma})$, which yields
\[
\int_{\Omega_{E}}{H_E} \nabla w  \cdot \nabla w dx=\int_{\Omega}{H} \nabla (w|_{\Omega})  \cdot \nabla (w|_{\Omega}) dx =0,
\]
for any $w=w^g,\,g\in\dot{H}^{- 1/2}_{\diamond}(\overline{\Sigma_{E}})$. Then, the proof of \eqref{extension} ends by noticing that we have $\mathrm{supp}H_E \cap \Sigma_E = \emptyset$. \end{remark}
\medskip
\vspace{5pt}
Next, we turn to consider the orthotropic case of Proposition \ref{Local_injectivity} (that is, $\psi=\phi=2\pi$, $\tilde k=\frac{1}{\mathrm{tan}\theta}$, $a=0$, and $b\neq 0$).
Then, the determinant of equation (\ref{condition2}) is determined by
\[
\mathrm{(Determinant)}=-b\,\tilde k.
\]
Hence, the determinant is not zero except $\theta=\frac{\pi}{2}$, which implies the following lemma and collorary.
\begin{lemma}\label{Local_injectivity_oth}
For $r \in (0,1]$ and $\theta \in (0, \pi)$, define
\[
V_{\mathcal{P}_{r,\theta}}^{c,ort}:=\left\{ H = \left(
\begin{array}{cc}
h_{1}& 0 \\
0 & h_{2} \\
\end{array}
\right): \Omega \to \mathbb{R}^{2 \times 2} \,:
\mbox{$H$ is  $\mathrm{constant}$ for each $C\in \mathcal{P}_{r,\theta}$}\right\},
\]
and $V_{\mathcal{P}_{r,\theta}}^{c,ort,+}:= L^{\infty}_{+}(\Omega) \cap V_{\mathcal{P}_{r,\theta}}^{c,oth}$.
Let $\theta\in(0,\pi) \setminus \{ \frac{\pi}{2}\}$, $r_{0} \in (0,1]$, and $\gamma_0 \in \mbox{\rm int}(V_{\mathcal{P}_{r_0,\theta}}^{c,ort,+})$.
Then, assuming that $r,r_{0} \in (0,1]$ satisfy the irrational condition, that is, $\frac{r}{r_0}\not\in{\mathbb Q}$, we have 
\begin{equation}
\min\{ \left\|F^{\prime}(\gamma_{0})[H]\right\| \,: H \in V_{\mathcal{P}_{r,\theta}}^{c, oth}, \ \left\| H \right\|_{\infty}=1 \}>0.
\end{equation}
\end{lemma}
\begin{corollary}\label{Probabilistic LLS ort}
Fix $\theta\in(0,\pi) \setminus \{ \frac{\pi}{2}\}$, $r_0 \in (0, \pi]$, and $\gamma_0 \in \mbox{\rm int}(V_{\mathcal{P}_{r_0,\theta}}^{c,ort,+})$.
For $r \in (0,1]$, let $Q^{p, ort}_{r}$ be a proposition stated as follows: 
\begin{equation}
\left\{
\begin{array}{ll}
{}^\exists\epsilon>0:\\
\quad\,\, \,{}^\forall\gamma\in\mbox{\rm int}(V_{\mathcal{P}_{r,\theta}}^{c, ort, +})\,\,\mbox{\rm with}\,\,\Vert\gamma-\gamma_0\Vert_\infty\le\epsilon,\\
\quad\,\,\, {}^\exists\delta=\delta(\gamma)>0\,\,\mbox{\rm with}\,\,B_\delta(\gamma)\subset\mbox{\rm int}(V_{\mathcal{P}_{r,\theta}}^{c,ort,+})
\\
\qquad\,\,\,\,\mbox{\rm such that}\\
\qquad\qquad\Vert\tau-\sigma\Vert_\infty\le C
\Vert F(\tau)-F(\sigma)\Vert,\,\,\tau,\,\sigma\in B_\delta(\gamma)\\
\qquad\qquad\qquad\qquad\quad\,\,\,\mbox{\rm with a constant}\,\,C=C(\gamma)>0,
\end{array}
\right.
\end{equation}
Furthermore, let $\mathcal{E}^{p,ort}=\{ r \in (0,1] \,: Q_{r}^{p,ort} \ \text{\rm is true}\}$ be an event. 
Then,
\[
m(\mathcal{E}^{p,ort})=1,
\]
where $m$ is the two-dimensional Lebesgue measure.
\end{corollary}
\subsection{Trapezoidal decomposition}\label{Trapezoidal decomposition}
\begin{proposition} \label{Local_injectivity_trap}
Fix $r_{0}, q \in (0,1]$ satisfying the irrational condition, that is, $\frac{r_0}{q} \not \in{\mathbb Q}$, and $\gamma_0 \in \mbox{\rm int}(V_{\mathcal{L}_{r_0,\theta}}^{c,+})$.
Let $r,r_{0} \in (0,1]$ satisfy the irrational condition, that is, $\frac{r}{r_0}\not\in{\mathbb Q}$.
Assume that $\theta \in (0, \pi)$ satisfies
\begin{equation}
\pm \left<\left( \gamma_0|_{C}\right)^{-1} 
\left(
\begin{array}{c}
1 \\
0 \\
\end{array}
\right),
\left(
\begin{array}{c}
\frac{1}{\mathrm{tan}\theta} \\
1 \\
\end{array}
\right)
\right> \neq 0, \ \ \mbox{\rm for each cell $C \in \mathcal{L}_{r_0,\theta}$}.
\label{assumption_trape}
\end{equation}
Then, we have
\begin{equation}
\min\{\left\|F^{\prime}(\gamma_{0})[H]\right\| \,: H \in V_{\mathcal{T}_{r,\theta}}^{c}, \ \left\| H \right\|_{\infty}=1 \}>0.
\end{equation}
\end{proposition}
\begin{proof}
We basically follow the proof of Proposition \ref{Local_injectivity}. 
To begin with, we note that (\ref{integral identity}) gives us the integral identity 
\[
\int_{\Omega}H \nabla u  \cdot \nabla u dx = 0,
\]
for any unique solution $u \in H^{1}_{\diamond}(\Omega)$ of
\[
\nabla \cdot \gamma_0 \nabla   u = 0\,\,\text{in}\,\,\Omega, \ \ \partial_{ \gamma_{0}}  u= f\,\,\text{on}\,\,\partial \Omega, \,\,{f} \in \dot{H}^{- 1/2}_{\diamond}(\overline{\Sigma}).
\]
and some $H \in  V_{\mathcal{L}_{r_0,\theta}}^{c}$ with $\left\|H \right\|=1$.

By the extension argument (see Remark \ref{whole_domain2}), we can assume that
\[
\mathrm{supp}H \cap \Sigma = \emptyset,
\]
that is, $\mathrm{supp}H$ consists of several cells in $ {\mathcal{T}_{r,\theta}}^{c}$ without having the intersection with $\Sigma$.
Then, there exist two exposed lower corner points $z_{0}$ and $z_1$ of some cell $C_H \in \mathcal{T}_{r,\theta}$ with angle $\theta$ or $\pi - \theta$. Further, these points are included in the interior of a lateral cell $C_\gamma$ for $\gamma_0$ (see the left in Figure \ref{transformations_trape}).

\begin{figure}[h]
\hspace{0.0cm}
\center
\includegraphics[keepaspectratio, scale=0.6]{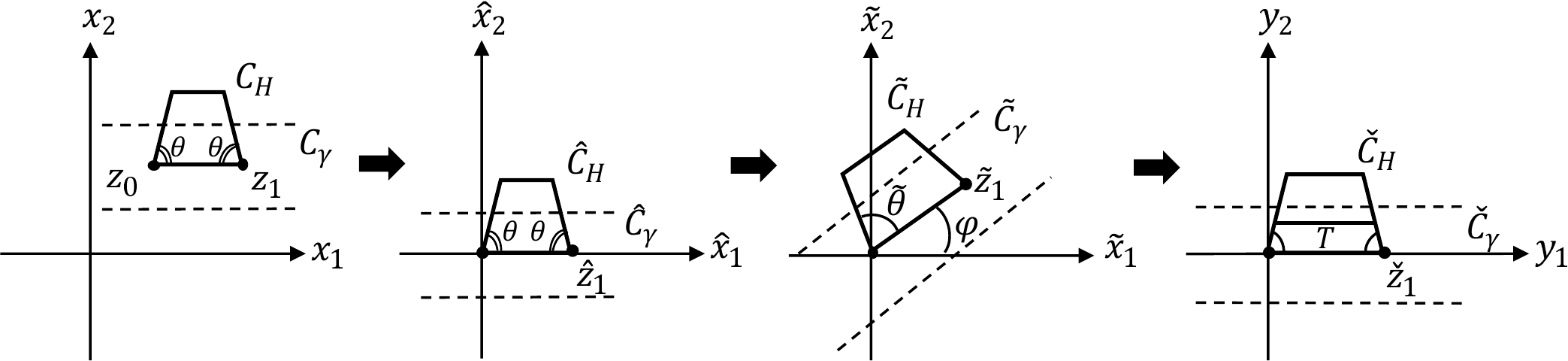}
\caption{Transformations.}
\label{transformations_trape}
\end{figure}

Now, by the argument in the proof of Proposition \ref{Local_injectivity} using the transformations (\ref{integral identity2})--(\ref{integral identity4}), the Runge approximation and the estimate for fundamental solutions (\ref{key integral}), we can show that
\[
I(\eta):= \int_{T}\check{H}\nabla_{y} \log|y-\eta| \cdot \nabla_{y}\log|y-\eta| dy, \ \mbox{for $\eta \in V\setminus T $}
\]
satisfies
\[
\sup_{\eta \in V \setminus T}|I(\eta)| < \infty,
\]
where $\check{H}$ is defined by the same way as (\ref{integral identity2})--(\ref{integral identity4}), and $T,\,V$ are defined as follows.
Namely, $T$ is a small closed isosceles trapezoid in a cell $\check{C}_H$ of $\mathrm{supp}\check{H}$ with the lower angle $\tilde \theta$ or $\pi - \tilde \theta$ at the origin and $\check z_1$, and $V$ is an open neighborhood of $T$ (see Figure \ref{transformations_trape} right). Here we can assume that by taking $T,\,V$ smaller if necessary, there exists a cell $\check{C}_{\gamma}$ for $\check{\gamma}_0$ such that $\check{\gamma}_0=I$ and $V\Subset\check{C}_\gamma$. Also, by the assumption (\ref{assumption_trape}), we have $\tilde{\theta} \neq \frac{\pi}{2}$. This is because $\tilde{\theta}$ is the angle between $D^{-1/2}R_{\psi}\left(
\begin{array}{c}
1 \\
0
\end{array}
\right)$ and $D^{-1/2}R_{\psi}\left(
\begin{array}{c}
k \\
1
\end{array}
\right),$
where $k=\frac{1}{\mathrm{tan}\theta}$ or $\frac{1}{\mathrm{tan}(\pi - \theta)}$.
Using Lemma \ref{Blow up of the integral_trap}, we have 
\[
\check H\bigl|_{\check C_H}=0,
\]
which is equivalent to $H\bigl|_{C_H}=0$, hence this contradicts to $\mathrm{supp}H\neq\emptyset$.
\end{proof}
\section{Core computations}
In this section, we provide the core computations obtaining some conditions for $H$ leading to $H=0$ which were skipped in the proofs of the conditional injectivity of the Fr\'{e}chet derivative in Subsection \ref{Parallelogram decomposition} and Subsection \ref{Trapezoidal decomposition}.
We first prove the following lemma applied in Subsection \ref{Parallelogram decomposition} which gives two independent conditions for $H$ obtained by approaching the singular point of a fundamental solution for $\gamma_0$ to the edge of a single corner and then to its vertex.

\begin{lemma}\label{Blow up of the integral}
Let $\epsilon >0$, $k \in \mathbb{R}$, and let $H=\left(
\begin{array}{cc}
h_{11}& h_{12} \\
h_{12}& h_{22} \\
\end{array}
\right)$ be a symmetric constant matrix.
Define the closed trapezoid $T:=\{ k y_{2} \leq y_{1} \leq \epsilon, \ 0 \leq y_2 \leq \epsilon \} \subset [0, \infty) \times (-\infty, 0]$ (see Figure \ref{trapezoid_int}). 
Let $V\subset{\mathbb R}^2$ be an open neighborhood of $T$ and define
\[
I(\eta) := \int_{T}H\nabla_{y}\log|y-\eta| \cdot \nabla_{y}\log|y-\eta| dy,\,\,
\eta\notin T.\]
Then, 
\begin{equation}
\sup_{\eta \in V \setminus T}|I(\eta)| < \infty.
\label{assumption_blow_up}
\end{equation}
implies
\begin{equation}
 \left \{
\begin{array}{l}
h_{11}+h_{22}=0\\
kh_{11} + h_{12}=0.
\end{array}
\right. \label{two_conditions}
\end{equation}
\end{lemma}
\begin{figure}[h]
\hspace{0.0cm}
\center
\includegraphics[keepaspectratio, scale=0.7]{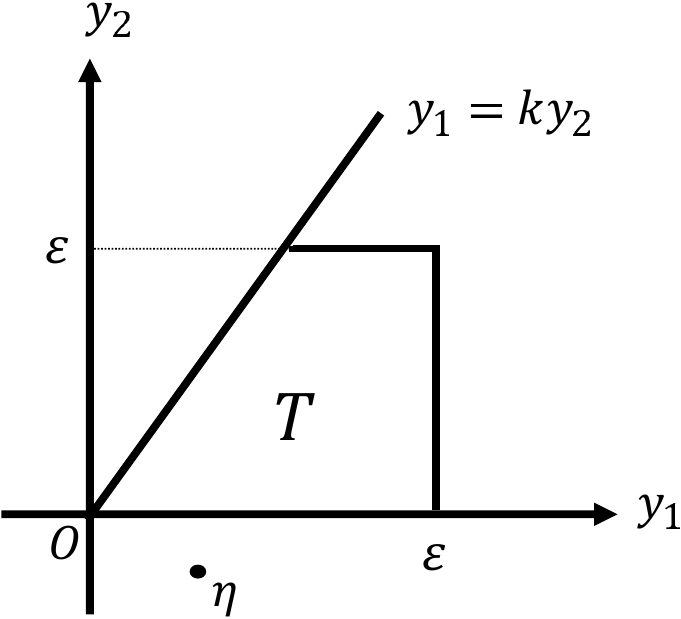}
\caption{Integral on trapezoid.}
\label{trapezoid_int}
\end{figure}
\begin{proof} Let us assume \eqref{assumption_blow_up}. For $\eta=(\eta_1,\eta_2)^T\in V\setminus T$, express $I(\eta)$ in the form
\begin{equation}
\begin{split}
& I(\eta)=\int_{T}H\nabla \log|y-\eta| \cdot \nabla \log|y-\eta| dy
\\
& = \int_{T}\frac{h_{11}(y_{1}-\eta_{1})^{2}+h_{22}(y_{2}-\eta_{2})^{2}+2h_{12}(y_{1}-\eta_{1})(y_{2}-\eta_{2})}{\{(y_{1}-\eta_{1})^{2}+(y_{2}-\eta_{2})^{2} \}^{2}} dy. \label{Ieta}
\end{split}
\end{equation}
Since we approach $\eta$ to the bottom edge, $\eta_{2}<0$, $\eta_{1}-k\eta_{2} \neq 0$.
\par
By the change of variables given as $t=y_{1}-\eta_{1},\,s=y_{2}-\eta_{2}$ in (\ref{Ieta}) and further making the change of variables $t=su$, we have
\begin{equation}
\begin{split}
& I(\eta)=\int_{0}^{\epsilon}\int_{ky_{2}}^{\epsilon} \frac{h_{11}(y_{1}-\eta_{1})^{2}+h_{22}(y_{2}-\eta_{2})^{2}+2h_{12}(y_{1}-\eta_{1})(y_{2}-\eta_{2})}{\{(y_{1}-\eta_{1})^{2}+(y_{2}-\eta_{2})^{2} \}^{2}} dy_{1}dy_{2}
\\
& = \int^{\epsilon - \eta_{2}}_{-\eta_{2}}\int^{\epsilon-\eta_{1}}_{k(s+\eta_{2})-\eta_{1}} \frac{h_{11}t^{2}+h_{22}s^{2}-2h_{12}ts}{\{t^{2}+s^{2}\}^{2}} dtds 
\\
& = \int^{\epsilon - \eta_{2}}_{-\eta_{2}}\frac{1}{s}\int^{\frac{\epsilon-\eta_{1}}{s}}_{\frac{k(s+\eta_{2})-\eta_{1}}{s}} \frac{h_{11}u^{2}+h_{22}+2h_{12}u}{\{u^{2}+1 \}^{2}}du ds.
\end{split}
\end{equation}
Note that
\[
\frac{h_{11}u^{2}+h_{22}+2h_{12}u}{\{u^{2}+1 \}^{2}}=\frac{h_{11}+h_{22}}{2}\frac{d}{du}\mathrm{Arctan}u - \frac{h_{11}-h_{22}}{2}\frac{d}{du} \frac{u}{1+u^{2}}-h_{12}\frac{d}{du} \frac{1}{1+u^{2}}.
\]
Hence we have
\begin{equation}
\begin{split}
& I(\eta)=\frac{h_{11}+h_{22}}{2}\int^{\epsilon-\eta_{2}}_{-\eta_{2}} \frac{1}{s} \left( \mathrm{Arctan}\left(\frac{\epsilon-\eta_{1}}{s}\right) - \mathrm{Arctan}\left(\frac{ks+k\eta_{2}-\eta_{1}}{s}\right) \right) ds
\\
& - \frac{h_{11}-h_{22}}{2}\int^{\epsilon-\eta_{2}}_{-\eta_{2}} \frac{1}{s}  \left( \frac{\left(\frac{\epsilon-\eta_{1}}{s}\right)}{1+\left(\frac{\epsilon-\eta_{1}}{s}\right)^{2}} - \frac{\left(\frac{ks+k\eta_{2}-\eta_{1}}{s}\right)}{1+\left(\frac{ks+k\eta_{2}-\eta_{1}}{s}\right)^{2}} \right) ds
\\
& + h_{12}\int^{\epsilon-\eta_{2}}_{-\eta_{2}} \frac{1}{s}  \left( \frac{1}{1+\left(\frac{\epsilon-\eta_{1}}{s}\right)^{2}} - \frac{1}{1+\left(\frac{ks+k\eta_{2}-\eta_{1}}{s}\right)^{2}} \right) ds. \label{integral1}
\end{split}
\end{equation}
For the further computation of (\ref{integral1}), we compute the integral
\begin{equation}
J(f):=\int^{\epsilon-\eta_{2}}_{-\eta_{2}} \frac{1}{s} f\left(\frac{\epsilon-\eta_{1}}{s}\right) ds - \int^{\epsilon-\eta_{2}}_{-\eta_{2}}\frac{1}{s} f\left(\frac{ks+k\eta_{2}-\eta_{1}}{s}\right)  ds,
\end{equation}
where $f(x)=\mathrm{Arctan}(x)$ or $\frac{x}{1+x^{2}}$ or $\frac{1}{1+x^{2}}$.
By the change of variables given as $v=\frac{\epsilon-\eta_{1}}{s},\,v=\frac{ks+k\eta_{2}-\eta_{1}}{s}$ and integrating by parts, we have
\begin{equation}
\begin{split}
& J(f)=-\int^{\frac{\epsilon-\eta_{1}}{\epsilon-\eta_{2}}}_{\frac{\epsilon-\eta_{1}}{-\eta_{2}}} \frac{v}{\epsilon - \eta_{1}}f(v) \frac{\epsilon-\eta_{1}}{v^{2}} dv + \int^{\frac{k\epsilon-\eta_{1}}{\epsilon-\eta_{2}}}_{\frac{\eta_{1}}{\eta_{2}}} \frac{v-k}{k\eta_{2} - \eta_{1}}f(v) \frac{k\eta_{2}-\eta_{1}}{(v-k)^{2}} dv
\\
& = - \int^{\frac{\epsilon-\eta_{1}}{\epsilon-\eta_{2}}}_{\frac{\epsilon-\eta_{1}}{-\eta_{2}}} \left( \frac{d}{dv}\log|v|\right) f(v) dv + \int^{\frac{k\epsilon-\eta_{1}}{\epsilon-\eta_{2}}}_{\frac{\eta_{1}}{\eta_{2}}} \left(\frac{d}{dv}\log|v-k|\right)f(v) dv
\\
& = - \left[ \log|v| f(v) \right]^{\frac{\epsilon-\eta_{1}}{\epsilon-\eta_{2}}}_{\frac{\epsilon-\eta_{1}}{-\eta_{2}}} + \int^{\frac{\epsilon-\eta_{1}}{\epsilon-\eta_{2}}}_{\frac{\epsilon-\eta_{1}}{-\eta_{2}}} \log|v| \frac{d}{dv}f(v) dv
\\
& + \left[ \log|v-k| f(v) \right]^{\frac{k\epsilon-\eta_{1}}{\epsilon-\eta_{2}}}_{\frac{\eta_{1}}{\eta_{2}}} - \int^{\frac{k\epsilon-\eta_{1}}{\epsilon-\eta_{2}}}_{\frac{\eta_{1}}{\eta_{2}}} \log|v-k| \frac{d}{dv}f(v) dv. \label{integral1-1}
\end{split}
\end{equation}
Since 
\[
\int_{-\infty}^{\infty}\log|v-a|\frac{d}{dv}f(v)dv < \infty, \ \ a \in \mathbb{R},
\]
for all the three cases $f(x)=\mathrm{Arctan}(x)$, $\frac{x}{1+x^{2}}$, $\frac{1}{1+x^{2}}$, we have 
\begin{equation}
\begin{split}
& J(f) =  \underbrace{-f\left(\frac{\epsilon-\eta_{1}}{\epsilon-\eta_{2}} \right) \log\left| \frac{\epsilon-\eta_{1}}{\epsilon-\eta_{2}}\right|}_{=O(1)} + f\left(\frac{\epsilon-\eta_{1}}{-\eta_{2}}\right) \underbrace{\log\left|\frac{\epsilon-\eta_{1}}{-\eta_{2}}\right|}_{\log\frac{1}{|\eta_{2}|}+O(1)}
\\
& 
+f\left(\frac{k\epsilon-\eta_{1}}{\epsilon-\eta_{2}}\right)
\underbrace{\log\left| \frac{k\epsilon-\eta_{1}}{\epsilon-\eta_{2}} - k \right|}_{=\log\left|\eta_{1}-k\eta_{2}\right|+O(1)} 
- f\left(\frac{\eta_{1}}{\eta_{2}}\right) \underbrace{\log\left|\frac{\eta_{1}}{\eta_{2}}-k\right|}_{=\log\left|\eta_{1}-k\eta_{2}\right|+\log\frac{1}{|\eta_{2}|}} + O(1)
\\
&
=\left\{ f\left(\frac{\epsilon-\eta_{1}}{-\eta_{2}}\right) - f\left(\frac{\eta_{1}}{\eta_{2}}\right) \right\} \log\frac{1}{|\eta_{2}|}
\\
& + \left\{ f\left(\frac{k\epsilon-\eta_{1}}{\epsilon-\eta_{2}}\right) - f\left(\frac{\eta_{1}}{\eta_{2}}\right) \right\} \log\left|\eta_{1}-k\eta_{2}\right|+ O(1). \label{integral2}
\end{split}
\end{equation}
Using (\ref{integral1}) and (\ref{integral2}), we have
\begin{equation} \label{after_comp}
\begin{split}
& I(\eta) = \Biggl[ \frac{h_{11}+h_{22}}{2} \left\{ \mathrm{Arctan}\left(\frac{\epsilon-\eta_{1}}{-\eta_{2}}\right) - \mathrm{Arctan}\left(\frac{\eta_{1}}{\eta_{2}}\right) \right\} 
\\
& - \frac{h_{11}-h_{22}}{2}\left\{ \frac{\left(\frac{\epsilon-\eta_{1}}{-\eta_{2}}\right)}{1+\left(\frac{\epsilon-\eta_{1}}{-\eta_{2}}\right)^{2}} - \frac{\left(\frac{\eta_{1}}{\eta_{2}}\right)}{1+\left(\frac{\eta_{1}}{\eta_{2}}\right)^{2}} \right\}
\\
& - h_{12}\left\{ \frac{1}{1+\left(\frac{\epsilon-\eta_{1}}{-\eta_{2}}\right)^{2}} - \frac{1}{1+\left(\frac{\eta_{1}}{\eta_{2}}\right)^{2}} \right\}\Biggr] \log\frac{1}{|\eta_{2}|}
\\
& + \Biggl[ \frac{h_{11}+h_{22}}{2} \left\{ \mathrm{Arctan}\left(\frac{k\epsilon-\eta_{1}}{\epsilon-\eta_{2}}\right) - \mathrm{Arctan}\left(\frac{\eta_{1}}{\eta_{2}}\right) \right\} 
\\
& - \frac{h_{11}-h_{22}}{2}\left\{ \frac{\left(\frac{k\epsilon-\eta_{1}}{\epsilon-\eta_{2}}\right)}{1+\left(\frac{k\epsilon-\eta_{1}}{\epsilon-\eta_{2}}\right)^{2}} - \frac{\left(\frac{\eta_{1}}{\eta_{2}}\right)}{1+\left(\frac{\eta_{1}}{\eta_{2}}\right)^{2}} \right\}
\\
& - h_{12}\left\{ \frac{1}{1+\left(\frac{k\epsilon-\eta_{1}}{\epsilon-\eta_{2}}\right)^{2}} - \frac{1}{1+\left(\frac{\eta_{1}}{\eta_{2}}\right)^{2}} \right\}\Biggr] \log|\eta_{1}-k\eta_{2}| + O(1).
\end{split}
\end{equation}
First, fix $0<\eta_1<<1$ and consider the asymptotic behavior of $I(\eta)$ as $\eta_2 \to -0$. Then, we have
\[
\begin{split}
& I(\eta) = \Biggl[ \frac{h_{11}+h_{22}}{2} \left\{ \mathrm{Arctan}\left(\frac{\epsilon-\eta_{1}}{-\eta_{2}}\right) - \mathrm{Arctan}\left(\frac{\eta_{1}}{\eta_{2}}\right) \right\} 
\\
& - \frac{h_{11}-h_{22}}{2}\left\{ \frac{\left(\frac{\epsilon-\eta_{1}}{-\eta_{2}}\right)}{1+\left(\frac{\epsilon-\eta_{1}}{-\eta_{2}}\right)^{2}} - \frac{\left(\frac{\eta_{1}}{\eta_{2}}\right)}{1+\left(\frac{\eta_{1}}{\eta_{2}}\right)^{2}} \right\}
\\
& - h_{12}\left\{ \frac{1}{1+\left(\frac{\epsilon-\eta_{1}}{-\eta_{2}}\right)^{2}} - \frac{1}{1+\left(\frac{\eta_{1}}{\eta_{2}}\right)^{2}} \right\}\Biggr] \log\frac{1}{|\eta_{2}|} + O(1)
\\
& = \frac{h_{11}+h_{22}}{2} \left\{ \mathrm{Arctan}\left(\frac{\epsilon-\eta_{1}}{-\eta_{2}}\right) - \mathrm{Arctan}\left(\frac{\eta_{1}}{\eta_{2}}\right) \right\} \log\frac{1}{|\eta_{2}|}
\\
& - \frac{h_{11}-h_{22}}{2}\underbrace{\left\{ \frac{-\left(\epsilon-\eta_{1}\right)\eta_2}{\eta_2^{2}+\left(\epsilon-\eta_{1}\right)^{2} } - \frac{\eta_1\eta_2}{\eta_1^2+\eta_2^2} \right\} \log\frac{1}{|\eta_{2}|}}_{=O(1)}
\\
& - h_{12}\underbrace{\left\{ \frac{\eta^2_2}{\eta_2^{2}+\left(\epsilon-\eta_{1}\right)^{2} } - \frac{\eta^2_2}{\eta_1^2+\eta_2^2} \right\} \log\frac{1}{|\eta_{2}|}}_{=O(1)} + O(1)
\end{split}
\]
\[
= \frac{h_{11}+h_{22}}{2} \left\{ \mathrm{Arctan}\left(\frac{\epsilon-\eta_{1}}{-\eta_{2}}\right) - \mathrm{Arctan}\left(\frac{\eta_{1}}{\eta_{2}}\right) \right\} \log\frac{1}{|\eta_{2}|} + O(1).
\]
Here, by the assumption (\ref{assumption_blow_up}), the integral $I(\eta)$ should not blow up as $\eta_2 \to -0$. Hence, we have the following necessary condition
\[
\lim_{\eta_2 \to -0}\frac{h_{11}+h_{22}}{2} \underbrace{ \left\{ \mathrm{Arctan}\left(\frac{\epsilon-\eta_{1}}{-\eta_{2}}\right) - \mathrm{Arctan}\left(-\frac{\eta_{1}}{\eta_{2}}\right) \right\}}_{\to \pi} = 0
\]

\begin{equation}
\Leftrightarrow h_{11}+h_{22}=0. \label{cond1}
\end{equation}
\par
Next, using (\ref{cond1}) and applying $\lim_{\eta_2\rightarrow-0}$ to the first term of (\ref{after_comp}), we consider the asymptotic behavior of $I(\eta_1,-0):=\lim_{\eta_2\rightarrow-0}I(\eta)$ as $\eta_1 \to +0$.
Then, we have
\[
I(\eta_1, -0) 
= -\Biggl[\frac{h_{11}-h_{22}}{2}\frac{\left(\frac{k\epsilon-\eta_{1}}{\epsilon}\right)}{1+\left(\frac{k\epsilon-\eta_{1}}{\epsilon}\right)^{2}}
+ h_{12} \frac{1}{1+\left(\frac{k\epsilon-\eta_{1}}{\epsilon}\right)^{2}} \Biggr] \log|\eta_{1}|  + O(1).
\]
Here, by the assumption (\ref{assumption_blow_up}), the integral $I(\eta_1, -0)$ should not blow up as $\eta_1 \to +0$. Hence, we have the following necessary condition 
\[
\lim_{\eta_1 \to +0}\Biggl[\underbrace{\frac{h_{11}-h_{22}}{2}}_{=h_{11}}\underbrace{\frac{\left(\frac{k\epsilon-\eta_{1}}{\epsilon}\right)}{1+\left(\frac{k\epsilon-\eta_{1}}{\epsilon}\right)^{2}}}_{=\frac{k}{1+k^2}}
+ h_{12} \underbrace{\frac{1}{1+\left(\frac{k\epsilon-\eta_{1}}{\epsilon}\right)^{2}}}_{=\frac{1}{1+k^2}} \Biggr] = 0
\]
\begin{equation}
\Leftrightarrow kh_{11}+h_{12}=0. \label{cond2}
\end{equation}
Thus we have obtained \eqref{two_conditions}
\end{proof}

If we could have used two corners, then we might have three independent conditions which immediately implies $H=0$. By using the extension argument, this is indeed the case for our trapezoidal decomposition 
for the perturbative conductivity $H$. More precisely, we have following lemma.

\begin{lemma}\label{Blow up of the integral_trap}
Let $\theta \in (0, \pi) \setminus \{ \frac{\pi}{2}\}$, and let $H=\left(
\begin{array}{cc}
h_{11}& h_{12} \\
h_{12}& h_{22} \\
\end{array}
\right)$ be a symmetric constant matrix.
Denote by $T$ a closed isosceles trapezoid  including two corners. Let them be the origin and $z_1$ on $y_1$-axis with the lower angle $\theta$ as in Figure \ref{symm_trapezoid}.
Define 
\[
I(\eta) := \int_{T}H\nabla_{y}\log|y-\eta| \cdot \nabla_{y}\log|y-\eta| dy,\,\,\eta\notin T
\]
and let $V \subset \mathbb{R}^2$ be an open neighborhood of $T$. Then,
\begin{equation}
\sup_{\eta \in V \setminus T}|I(\eta)| < \infty.
\label{assumption_blow_up_trape}
\end{equation}
implies
\[
H =0.
\]
\end{lemma}

\begin{figure}[h]
\hspace{0.0cm}
\center
\includegraphics[keepaspectratio, scale=0.6]{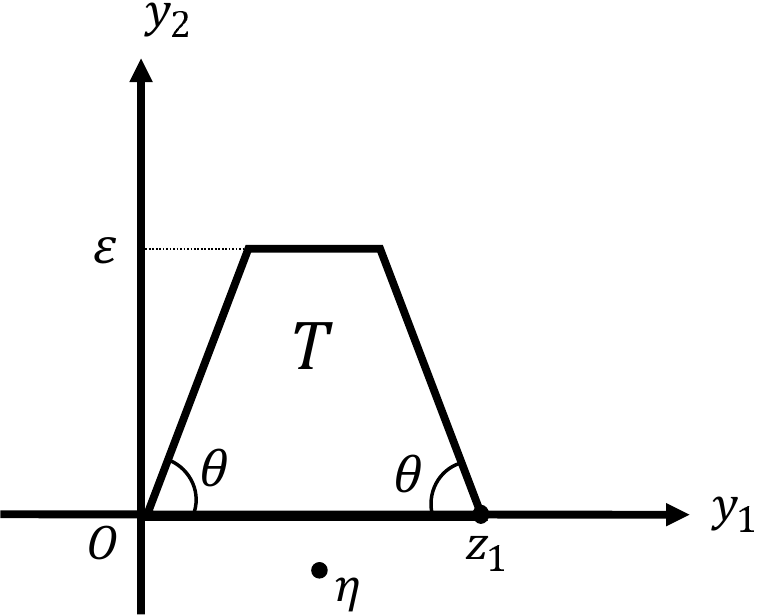}
\caption{Integral on isosceles trapezoid.}
\label{symm_trapezoid}
\end{figure}
\begin{proof}
Let us assume \eqref{assumption_blow_up_trape}. First, we let the singular point $\eta$ of $\log|y-\eta|$ approach to the corner at the origin. Then, from Lemma \ref{Blow up of the integral}, we have
\begin{equation}
\left \{
\begin{array}{l}
h_{11}+h_{22}=0,\\
kh_{11} + h_{12}=0,
\end{array}
\right. \label{condition_first}
\end{equation}
where $k=\frac{1}{\mathrm{tan}\theta}$. 
\par
\begin{figure}[h]
\hspace{0.0cm}
\center
\includegraphics[keepaspectratio, scale=0.55]{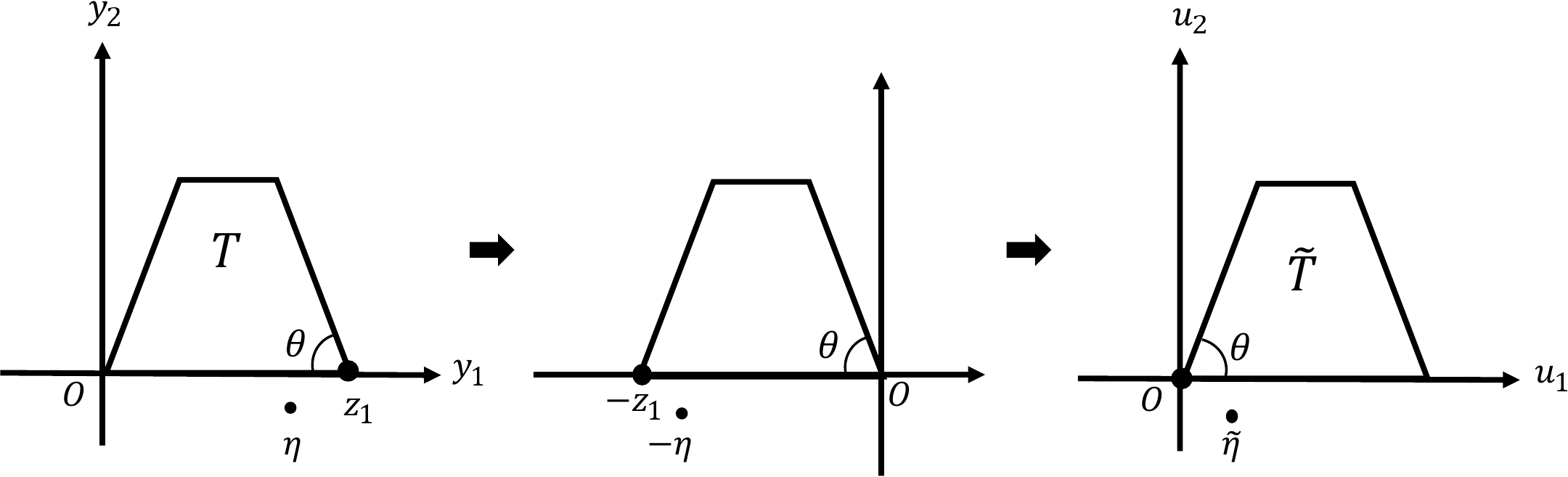}
\caption{Transformation.}
\label{trapezoid_trans}
\end{figure}
Next, we let the singular point $\eta$ of $\log|y-\eta|$ approach to the corner at $z_1$.
By the change of variables given as
$y=Au + z_1
$ with $A=\left(
\begin{array}{cc}
-1 & 0 \\
0 & 1 \\
\end{array}
\right)$ (see Figure \ref{trapezoid_trans}), we have
\[
\begin{split}
& I(\eta) = 
\int_{T}H\nabla_{y}\log|y-\eta| \cdot \nabla_{y}\log|y-\eta| dy 
\\
& =
\int_{\tilde{T}}
A^{T}HA
\nabla_{u}\log|Au + z_1 -\eta| \cdot \nabla_{u}\log|Au + z_1 -\eta| du
\\
& = 
\int_{\tilde{T}}
\left(
\begin{array}{cc}
h_{11} & -h_{12} \\
-h_{12} & h_{22} \\
\end{array}
\right)
\nabla_{u}\log|u - \tilde{\eta} | \cdot \nabla_{u}\log|u - \tilde{\eta} | du,
\end{split}
\]
where 
\[
\tilde{T}:=A(T-z_1), \ \ \ \tilde{\eta}:=A(\eta - z_1).
\]
Then, from Lemma \ref{Blow up of the integral}, we have
\[
\left \{
\begin{array}{l}
h_{11}+h_{22}=0,\\
kh_{11} - h_{12}=0.
\end{array}
\right.
\]
This together with (\ref{condition_first}) yields
\[
2kh_{11}=0.
\]
By the assumption $\theta \in (0, \pi) \setminus \{ \frac{\pi}{2}\}$, we have $k \neq 0$ and hence $h_{11}=0$, which immediately implies $h_{12}=h_{22}=0$. Thus, we have proved $H=0$.
\end{proof}

\section{Concluding remarks}

In this paper, we showed the probabilistic local Lipschitz stability for EIT in parallelogram- or trapezoid-based decomposed domains with piecewise constant anisotropic conductivities. As a result we also gave the probabilistic local 
recovery for the aforementioned EIT. The core computations of proving these are Propositions \ref{Local_injectivity} and \ref{Local_injectivity_trap}, and the key ingredients for its proof are the following:

\begin{itemize}

\item[(i)]
We used different grid-based decompositions of $\Omega$ for the background conductivity and perturbative conductivity. Upon using these decompositions, we can expose some corner points of a cell of the support of the perturbative conductivity inside the interior of a cell of the background conductivity by using the extension argument. This allowed us to conduct the core computations deriving the perturbative conductivity $H=0$, that is the injectivity of the Fr\'echet derivative $F'$ of the forward operator $F$. 
\item[(ii)]
In proving $H=0$ from the integral identity testing $H$ with solutions of the conductivity equation, we used the Runge approximation theorem to approximate the solutions by fundamental solutions of the background conductivity equation and targeted the exposed corner with the singular points of the fundamental solutions.
\end{itemize}

\noindent
Generalizations of our results that are worth exploring include
\begin{itemize}
\item As a natural extension of the work in this paper, an important study with immediate results is that of the three-dimensional case.
\item Generalizing the main results to the piecewise smooth case.
\item Showing the injectivity of the Fr\'{e}chet derivative 
with a joint, single domain decomposition, which would imply avoiding the irrationality and angle conditions.
\item Using other decompositions for $\Omega$ such as one with pentagons.
\end{itemize}

For all of these, we need to have a fundamental solution or a singular solution with a local pointwise estimate from below for its gradient, which will be our future works and opened to anybody who is interested in this issue.

\section*{Acknowledgments}
The first author was supported by the Simons Foundation under the MATH + X program, the National Science Foundation under grant DMS-2108175, and the corporate members of the Geo-Mathematical Imaging Group at Rice University.
The second author was supported by Grant-in-Aid for JSPS Fellows (No.21J00119), Japan Society for the Promotion of Science.
The third author was  partially supported by the Ministry of Science and Technology of Taiwan (Grant No. MOST 108-2115-M-006-018-MY3).
The fourth author was supported by JSPS KAKENHI (Grant No. JP19K03554). The fifth author was supported by Start-up Research Grant SRG/2021/001432  from the Science and Engineering Research Board, Government of India.

\bibliographystyle{plain}
\bibliography{Local recovery of a piecewise constant anisotropic conductivity in EIT on domains with exposed corners.bbl}

\end{document}